\documentclass[onefignum,onetabnum]{siamart190516}

\usepackage{lipsum}
\usepackage{amsfonts}
\usepackage{graphicx}
\usepackage{enumerate}
\usepackage{epstopdf}
\usepackage{microtype}
\usepackage{algorithmic}
\ifpdf
  \DeclareGraphicsExtensions{.eps,.pdf,.png,.jpg}
\else
  \DeclareGraphicsExtensions{.eps}
\fi

\newsiamremark{remark}{Remark}
\newsiamremark{hypothesis}{Hypothesis}
\crefname{hypothesis}{Hypothesis}{Hypotheses}
\newsiamthm{claim}{Claim}

\headers{Polygonal Digraphs}{T.~R.~Cameron, H.~T.~Hall, B.~Small, and A.~Wiedemann}

\title{On Digraphs with Polygonal Restricted Numerical Range\thanks{Submitted to the editors DATE.
\funding{This work was partially funded by the AMS-Simons Travel Grant Program.}}}

\author{Thomas R. Cameron\thanks{Department of Mathematics, Penn State Behrend, Erie, PA 
  (\email{trc5475@psu.edu}).}
\and H.~Tracy Hall\thanks{Hall Labs LLC, Provo, UT 
  (\email{H.Tracy@gmail.com}).}
\and Ben Small\thanks{(\email{bentsm@gmail.com}).}
\and Alexander Wiedemann\thanks{Department of Mathematics and Computer Science, Davidson College, Davidson, NC. Present address: Department of Mathematics, Randolph-Macon College, Ashland, VA (\email{alexanderwiedemann@rmc.edu}).}}

\usepackage{amsopn}
\newlength\myindent
\makeatletter
\newcommand{\oset}[3][0ex]{%
                \mathrel{\mathop{#3}\limits^{
                                                \vbox to#1{\kern-0.7\ex@
                                                                \hbox{$\scriptstyle#2$}\vss}}}}
\makeatother
\newcommand{\djoin}{\oset{\rightarrow}{\vee}}
\newcommand{\bdjoin}{\oset{\leftrightarrow}{\vee}}

\newcommand*\conj[1]{\overline{#1}}
\newcommand{\iu}{{i\mkern1mu}}
\newcommand\abs[1]{\left|#1\right|}
\newcommand\norm[1]{\left\Vert#1\right\Vert}
\newcommand\re[1]{\operatorname{Re}\left(#1\right)}

\newcommand\conv[1]{\operatorname{conv}\left(#1\right)}
\newcommand\spn[1]{\operatorname{span}\left(#1\right)}
\DeclareMathOperator{\diag}{diag}

\input{drawing}

\pgfkeys{/stored graphs/case1/.cd,
  x min/.initial=-0.999999791383765,
  x max/.initial=1.0,
  y min/.initial=-0.866025372082436,
  y max/.initial=0.8660255217554294,
  graph code/.code={
    \node (v0) at (-0.999999791383765, -1.0199322869699905e-07) {};
    \node (v1) at (-0.49999975413086595, -0.866025372082436) {};
    \node (v2) at (0.4999999627471009, -0.866025372082436) {};
    \node (v3) at (1.0, 1.6027507366671279e-07) {};
    \node (v4) at (0.4999996647239081, 0.8660255217554294) {};
    \node (v5) at (-0.500000081956378, 0.866025164127598) {};

    \graph {
      (v0); (v1); (v2); (v3); (v4); (v5);
      v0 -> v1;
      v1 -> v2;
      v2 -> v3;
      v3 -> v4;
      v4 -> v5;
      v5 -> v0;
    };
  }
}

\pgfkeys{/pgf/data/new group=case1-rnr}
\pgfdata[store in group=case1-rnr]{
  x, y
  0.49999999999999944, -0.8660254037844382
  1.4999999999999998, -0.8660254037844384
  2.0000000000000004, 0.0
  1.4999999999999998, 0.8660254037844384
  0.49999999999999944, 0.8660254037844382
}

\pgfkeys{/pgf/data/new group=case1-eig}
\pgfdata[store in group=case1-eig]{
  x, y
  0.49999999999999944, 0.8660254037844382
  0.49999999999999944, -0.8660254037844382
  2.0000000000000004, 0.0
  1.4999999999999998, 0.8660254037844384
  1.4999999999999998, -0.8660254037844384
}

\pgfkeys{/stored graphs/case2/.cd,
  x min/.initial=-1.0,
  x max/.initial=0.6854102419673151,
  y min/.initial=-0.5706339351295915,
  y max/.initial=0.5706338910600649,
  graph code/.code={
    \node (v0) at (0.014589783050400188, 0.5706338910600649) {};
    \node (v1) at (-0.40000000106687733, -2.7521080106309955e-08) {};
    \node (v2) at (0.014589844937823521, -0.5706339351295915) {};
    \node (v3) at (0.6854102419673151, -0.3526711083671352) {};
    \node (v4) at (0.6854101311113385, 0.35267121440446464) {};
    \node (v5) at (-1.0, -3.4446722641255056e-08) {};

    \graph {
      (v0); (v1); (v2); (v3); (v4); (v5);
      v0 -> { v2, v3, v4, v5 };
      v1 -> { v0, v3, v4, v5 };
      v2 -> { v1, v3, v4, v5 };
      v3 <-> v4;
    };
  }
}

\pgfkeys{/pgf/data/new group=case2-rnr}
\pgfdata[store in group=case2-rnr]{
  x, y
  0.0, 0.0
  4.5, -0.8660254037844388
  4.5, 0.8660254037844388
}

\pgfkeys{/pgf/data/new group=case2-eig}
\pgfdata[store in group=case2-eig]{
  x, y
  4.5, 0.8660254037844388
  4.5, -0.8660254037844388
  0.0, 0.0
  3.0, 0.0
  2.0, 0.0
}

\pgfkeys{/stored graphs/case3/.cd,
  x min/.initial=-1.0,
  x max/.initial=1.0,
  y min/.initial=-0.5,
  y max/.initial=0.5,
  graph code/.code={
    \node (v0) at (-1.0, 0.5) {};
    \node (v1) at (-1.0, -0.5) {};
    \node (v2) at (1.0, -0.5) {};
    \node (v3) at (1.0, 0.5) {};
    \node (v4) at (0.0, -0.5) {};
    \node (v5) at (0.0, 0.5) {};

    \graph {
      (v0); (v1); (v2); (v3); (v4); (v5);
      v0 <-> v1;
      v0 -> { v4, v5 };
      v1 -> { v4, v5 };
      v2 -> v4;
      v3 -> v5;
      v4 -> v3;
      v5 -> v2;
    };
  }
}

\pgfkeys{/pgf/data/new group=case3-rnr}
\pgfdata[store in group=case3-rnr]{
  x, y
  3.999999999999999, 0.0
  1.0, -1.0
  1.0, -0.2620444585260786
  0.9999999999999999, 0.9999999999999999
}

\pgfkeys{/pgf/data/new group=case3-eig}
\pgfdata[store in group=case3-eig]{
  x, y
  3.999999999999999, 0.0
  1.0000000000000007, 1.0000000000000002
  1.0000000000000007, -1.0000000000000002
  2.0000000000000018, 2.9074083778446188e-08
  2.0000000000000018, -2.9074083778446188e-08
}

\pgfkeys{/stored graphs/circ-normal/.cd,
  x min/.initial=-0.999999791383765,
  x max/.initial=1.0,
  y min/.initial=-0.866025372082436,
  y max/.initial=0.8660255217554294,
  graph code/.code={
    \node (v0) at (-0.999999791383765, -1.0199322869699905e-07) {};
    \node (v1) at (-0.49999975413086595, -0.866025372082436) {};
    \node (v2) at (0.4999999627471009, -0.866025372082436) {};
    \node (v3) at (1.0, 1.6027507366671279e-07) {};
    \node (v4) at (0.4999996647239081, 0.8660255217554294) {};
    \node (v5) at (-0.500000081956378, 0.866025164127598) {};

    \graph {
      (v0); (v1); (v2); (v3); (v4); (v5);
      v0 -> { v1, v2 };
      v1 -> { v2, v3 };
      v2 -> { v3, v4 };
      v3 -> { v4, v5 };
      v4 -> { v0, v5 };
      v5 -> { v0, v1 };
    };
  }
}

\pgfkeys{/pgf/data/new group=circ-normal-rnr}
\pgfdata[store in group=circ-normal-rnr]{
  x, y
  2.0000000000000004, 0.0
  2.0000000000000027, -1.7320508075688783
  3.0000000000000013, 0.0
  2.0000000000000027, 1.7320508075688783
}

\pgfkeys{/pgf/data/new group=circ-normal-eig}
\pgfdata[store in group=circ-normal-eig]{
  x, y
  2.0000000000000027, 1.7320508075688783
  2.0000000000000027, -1.7320508075688783
  3.0000000000000013, 0.0
  2.0000000000000004, 0.0
  3.000000000000001, 0.0
}

\pgfkeys{/stored graphs/bal-non-normal/.cd,
  x min/.initial=-0.999999791383765,
  x max/.initial=1.0,
  y min/.initial=-0.866025372082436,
  y max/.initial=0.8660255217554294,
  graph code/.code={
    \node (v0) at (-0.999999791383765, -1.0199322869699905e-07) {};
    \node (v1) at (-0.49999975413086595, -0.866025372082436) {};
    \node (v2) at (0.4999999627471009, -0.866025372082436) {};
    \node (v3) at (1.0, 1.6027507366671279e-07) {};
    \node (v4) at (0.4999996647239081, 0.8660255217554294) {};
    \node (v5) at (-0.500000081956378, 0.866025164127598) {};

    \graph {
      (v0); (v1); (v2); (v3); (v4); (v5);
      v0 -> v2;
      v1 -> v2;
      v2 -> { v3, v4 };
      v3 -> v5;
      v4 -> v5;
      v5 -> { v0, v1 };
    };
  }
}

\pgfkeys{/pgf/data/new group=bal-non-normal-rnr}
\pgfdata[store in group=bal-non-normal-rnr]{
  x, y
  3.0000000000000013, 0.0
  2.5304621357072365, -0.4695378642927636
  1.6657594954002393, -1.3342377289963236
  1.6648560135589714, -1.335132936582553
  1.6630600027954376, -1.3368962128403274
  1.6595109955466754, -1.3403171036727366
  1.6525788103951968, -1.3467582013661914
  1.6393303347055812, -1.3581960324808229
  1.6268192298621935, -1.3679674935820958
  1.6149542830414536, -1.3763299395433504
  1.603654056449479, -1.3834907634865332
  1.5928457659077373, -1.3896176848414263
  1.5824641689765526, -1.394846689310353
  1.5724505171694478, -1.3992881692114154
  1.5627515953368412, -1.403031681132485
  1.553318853591903, -1.4061496359965266
  1.5441076275761172, -1.4087001588403598
  1.5350764380120476, -1.4107292964885563
  1.526186358321247, -1.412272706350643
  1.5174004383756947, -1.4133569251801434
  1.5, -1.4142135623730954
  1.4825995616243048, -1.413356925180143
  1.4738136416787533, -1.4122727063506437
  1.464923561987952, -1.4107292964885554
  1.4558923724238835, -1.4087001588403592
  1.4466811464080978, -1.406149635996526
  1.437248404663158, -1.4030316811324863
  1.4275494828305517, -1.399288169211416
  1.4175358310234474, -1.3948466893103528
  1.4071542340922611, -1.3896176848414248
  1.3963459435505226, -1.3834907634865323
  1.3850457169585473, -1.3763299395433513
  1.373180770137807, -1.3679674935820962
  1.3606696652944192, -1.358196032480822
  1.3474211896048065, -1.346758201366192
  1.3333333333333326, -1.3333333333333326
  1.318292569214078, -1.3175201238605179
  1.3021737898626478, -1.2988141994379094
  1.2936681382950797, -1.2881837224414274
  1.284841515284735, -1.2765792617931129
  1.275675933969064, -1.2638946582665764
  1.266153394864861, -1.2500104488099943
  1.2562562469770926, -1.2347922347093307
  1.2459676758865592, -1.218088923865655
  1.2352723526351967, -1.1997308788914167
  1.2241572840060964, -1.1795280277572262
  1.2126129117612874, -1.1572680295028435
  1.200634514692617, -1.1327146373440222
  1.1882239714501044, -1.1056064689450928
  1.1753919415842735, -1.0756564820111374
  1.1688242428093663, -1.0595192035983514
  1.1621605133295807, -1.0425525646077243
  1.1554058152340048, -1.0247140003940964
  1.1485663440025458, -1.0059597822269557
  1.1416495517198526, -0.9862452279931899
  1.13466427536512, -0.9655249693529713
  1.1276208684128064, -0.9437532828377024
  1.1205313333747156, -0.9208844924294417
  1.1134094522187121, -0.896873450900019
  1.1062709108203563, -0.8716761065155709
  1.0991334127570231, -0.8452501605189762
  1.0920167768670699, -0.8175558189729492
  1.0849430121112618, -0.7885566399681251
  1.077936362448229, -0.7582204737756278
  1.0710233137464824, -0.7265204891863732
  1.0642325542964317, -0.6934362740150679
  1.0575948803664128, -0.6589549916162223
  1.0511430385858742, -0.6230725684225068
  1.0449114978562786, -0.585794880254741
  1.038936145092373, -0.5471388978876823
  1.0332539014590736, -0.5071337456476298
  1.0279022589177944, -0.46582162135374455
  1.0229187407844158, -0.4232585224684546
  1.0183402944941513, -0.37951472268695596
  1.014202629631158, -0.33467494609209464
  1.010539519178847, -0.2888381929638178
  1.0073820864638532, -0.24211718259212597
  1.0047581039371782, -0.19463739381939507
  1.002691332296601, -0.14653570287325715
  1.0012009291048933, -0.09795863919749687
  1.000300954726241, -0.04906030187381947
  0.9999999999999998, 0.0
  1.0003009547262416, 0.04906030187381849
  1.0012009291048938, 0.09795863919749598
  1.0026913322965998, 0.14653570287325604
  1.0047581039371782, 0.19463739381939424
  1.0073820864638536, 0.2421171825921251
  1.0105395191788473, 0.2888381929638172
  1.0142026296311584, 0.3346749460920942
  1.0183402944941515, 0.3795147226869552
  1.0229187407844162, 0.42325852246845375
  1.0279022589177937, 0.4658216213537431
  1.0332539014590738, 0.5071337456476289
  1.0389361450923733, 0.5471388978876818
  1.044911497856278, 0.5857948802547399
  1.0511430385858738, 0.6230725684225062
  1.0575948803664124, 0.6589549916162215
  1.0642325542964317, 0.6934362740150675
  1.0710233137464815, 0.726520489186372
  1.077936362448228, 0.7582204737756266
  1.084943012111261, 0.7885566399681238
  1.0920167768670694, 0.8175558189729485
  1.0991334127570234, 0.845250160518976
  1.1062709108203568, 0.8716761065155711
  1.1134094522187117, 0.896873450900018
  1.1205313333747158, 0.9208844924294416
  1.127620868412807, 0.9437532828377022
  1.1346642753651206, 0.9655249693529719
  1.1416495517198535, 0.9862452279931906
  1.1485663440025442, 1.0059597822269548
  1.1554058152340034, 1.0247140003940949
  1.1621605133295811, 1.0425525646077243
  1.1688242428093665, 1.0595192035983518
  1.1753919415842726, 1.0756564820111365
  1.188223971450104, 1.1056064689450926
  1.2006345146926165, 1.1327146373440216
  1.212612911761287, 1.1572680295028432
  1.2241572840060964, 1.1795280277572262
  1.235272352635196, 1.1997308788914158
  1.245967675886558, 1.2180889238656545
  1.2562562469770926, 1.234792234709331
  1.2661533948648618, 1.2500104488099948
  1.2756759339690638, 1.2638946582665762
  1.2848415152847363, 1.2765792617931144
  1.2936681382950788, 1.2881837224414268
  1.3021737898626484, 1.2988141994379099
  1.318292569214076, 1.3175201238605159
  1.3333333333333337, 1.333333333333334
  1.3474211896048045, 1.3467582013661903
  1.3606696652944188, 1.3581960324808218
  1.3731807701378078, 1.3679674935820967
  1.3850457169585462, 1.3763299395433504
  1.3963459435505214, 1.3834907634865314
  1.4071542340922618, 1.3896176848414257
  1.4175358310234483, 1.3948466893103542
  1.4275494828305517, 1.3992881692114159
  1.437248404663158, 1.4030316811324863
  1.4466811464080984, 1.4061496359965266
  1.4558923724238833, 1.408700158840359
  1.4649235619879528, 1.4107292964885563
  1.4738136416787524, 1.4122727063506426
  1.4825995616243057, 1.4133569251801439
  1.5000000000000013, 1.4142135623730967
  1.5174004383756943, 1.4133569251801428
  1.526186358321248, 1.4122727063506437
  1.5350764380120483, 1.4107292964885567
  1.5441076275761165, 1.4087001588403598
  1.5533188535919022, 1.4061496359965262
  1.562751595336841, 1.4030316811324846
  1.572450517169448, 1.3992881692114159
  1.5824641689765526, 1.394846689310353
  1.5928457659077364, 1.3896176848414252
  1.6036540564494766, 1.3834907634865314
  1.6149542830414545, 1.3763299395433521
  1.626819229862194, 1.3679674935820962
  1.6393303347055812, 1.3581960324808229
  1.6525788103951944, 1.346758201366189
  1.6595109955466745, 1.340317103672736
  1.6630600027954385, 1.3368962128403283
  1.6648560135589716, 1.335132936582554
  2.132147586144324, 0.867852413855676
}

\pgfkeys{/pgf/data/new group=bal-non-normal-eig}
\pgfdata[store in group=bal-non-normal-eig]{
  x, y
  0.9999999999999998, 0.0
  3.0000000000000027, 0.0
  1.5000000000000009, 1.3228756555322971
  1.5000000000000009, -1.3228756555322971
  0.9999999999999998, 0.0
}

\pgfkeys{/stored graphs/4-cycle/.cd,
  x min/.initial=-1.0,
  x max/.initial=1.0,
  y min/.initial=-1.0,
  y max/.initial=1.0,
  graph code/.code={
    \node (v0) at (-1.0, 1.0) {};
    \node (v1) at (-1.0, -1.0) {};
    \node (v2) at (1.0, -1.0) {};
    \node (v3) at (1.0, 1.0) {};

    \graph {
      (v0); (v1); (v2); (v3);
      v0 -> v1;
      v1 -> v2;
      v2 -> v3;
      v3 -> v0;
    };
  }
}

\pgfkeys{/pgf/data/new group=4-cycle-rnr}
\pgfdata[store in group=4-cycle-rnr]{
  x, y
  1.0, -0.9999999999999999
  2.0000000000000018, 0.0
  1.0, 0.9999999999999999
}

\pgfkeys{/pgf/data/new group=4-cycle-eig}
\pgfdata[store in group=4-cycle-eig]{
  x, y
  2.0000000000000018, 0.0
  1.0, 0.9999999999999999
  1.0, -0.9999999999999999
}

\pgfkeys{/stored graphs/twin-splitting/.cd,
  x min/.initial=-1.0,
  x max/.initial=1.0,
  y min/.initial=-1.0,
  y max/.initial=1.0,
  graph code/.code={
    \node (v0) at (-1.0, -0.5) {};
    \node (v1) at (-0.5, -1.0) {};
    \node (v2) at (1.0, -1.0) {};
    \node (v3) at (1.0, 0.5) {};
    \node (v4) at (0.5, 1.0) {};
    \node (v5) at (-1.0, 1.0) {};

    \graph {
      (v0); (v1); (v2); (v3); (v4); (v5);
      v0 -> v2;
      v1 -> v2;
      v2 -> { v3, v4 };
      v3 -> v5;
      v4 -> v5;
      v5 -> { v0, v1 };
    };
  }
}

\pgfkeys{/pgf/data/new group=twin-splitting-rnr}
\pgfdata[store in group=twin-splitting-rnr]{
  x, y
  3.0000000000000013, 0.0
  2.5304621357072365, -0.4695378642927636
  1.6657594954002393, -1.3342377289963236
  1.6648560135589714, -1.335132936582553
  1.6630600027954376, -1.3368962128403274
  1.6595109955466754, -1.3403171036727366
  1.6525788103951968, -1.3467582013661914
  1.6393303347055812, -1.3581960324808229
  1.6268192298621935, -1.3679674935820958
  1.6149542830414536, -1.3763299395433504
  1.603654056449479, -1.3834907634865332
  1.5928457659077373, -1.3896176848414263
  1.5824641689765526, -1.394846689310353
  1.5724505171694478, -1.3992881692114154
  1.5627515953368412, -1.403031681132485
  1.553318853591903, -1.4061496359965266
  1.5441076275761172, -1.4087001588403598
  1.5350764380120476, -1.4107292964885563
  1.526186358321247, -1.412272706350643
  1.5174004383756947, -1.4133569251801434
  1.5, -1.4142135623730954
  1.4825995616243048, -1.413356925180143
  1.4738136416787533, -1.4122727063506437
  1.464923561987952, -1.4107292964885554
  1.4558923724238835, -1.4087001588403592
  1.4466811464080978, -1.406149635996526
  1.437248404663158, -1.4030316811324863
  1.4275494828305517, -1.399288169211416
  1.4175358310234474, -1.3948466893103528
  1.4071542340922611, -1.3896176848414248
  1.3963459435505226, -1.3834907634865323
  1.3850457169585473, -1.3763299395433513
  1.373180770137807, -1.3679674935820962
  1.3606696652944192, -1.358196032480822
  1.3474211896048065, -1.346758201366192
  1.3333333333333326, -1.3333333333333326
  1.318292569214078, -1.3175201238605179
  1.3021737898626478, -1.2988141994379094
  1.2936681382950797, -1.2881837224414274
  1.284841515284735, -1.2765792617931129
  1.275675933969064, -1.2638946582665764
  1.266153394864861, -1.2500104488099943
  1.2562562469770926, -1.2347922347093307
  1.2459676758865592, -1.218088923865655
  1.2352723526351967, -1.1997308788914167
  1.2241572840060964, -1.1795280277572262
  1.2126129117612874, -1.1572680295028435
  1.200634514692617, -1.1327146373440222
  1.1882239714501044, -1.1056064689450928
  1.1753919415842735, -1.0756564820111374
  1.1688242428093663, -1.0595192035983514
  1.1621605133295807, -1.0425525646077243
  1.1554058152340048, -1.0247140003940964
  1.1485663440025458, -1.0059597822269557
  1.1416495517198526, -0.9862452279931899
  1.13466427536512, -0.9655249693529713
  1.1276208684128064, -0.9437532828377024
  1.1205313333747156, -0.9208844924294417
  1.1134094522187121, -0.896873450900019
  1.1062709108203563, -0.8716761065155709
  1.0991334127570231, -0.8452501605189762
  1.0920167768670699, -0.8175558189729492
  1.0849430121112618, -0.7885566399681251
  1.077936362448229, -0.7582204737756278
  1.0710233137464824, -0.7265204891863732
  1.0642325542964317, -0.6934362740150679
  1.0575948803664128, -0.6589549916162223
  1.0511430385858742, -0.6230725684225068
  1.0449114978562786, -0.585794880254741
  1.038936145092373, -0.5471388978876823
  1.0332539014590736, -0.5071337456476298
  1.0279022589177944, -0.46582162135374455
  1.0229187407844158, -0.4232585224684546
  1.0183402944941513, -0.37951472268695596
  1.014202629631158, -0.33467494609209464
  1.010539519178847, -0.2888381929638178
  1.0073820864638532, -0.24211718259212597
  1.0047581039371782, -0.19463739381939507
  1.002691332296601, -0.14653570287325715
  1.0012009291048933, -0.09795863919749687
  1.000300954726241, -0.04906030187381947
  0.9999999999999998, 0.0
  1.0003009547262416, 0.04906030187381849
  1.0012009291048938, 0.09795863919749598
  1.0026913322965998, 0.14653570287325604
  1.0047581039371782, 0.19463739381939424
  1.0073820864638536, 0.2421171825921251
  1.0105395191788473, 0.2888381929638172
  1.0142026296311584, 0.3346749460920942
  1.0183402944941515, 0.3795147226869552
  1.0229187407844162, 0.42325852246845375
  1.0279022589177937, 0.4658216213537431
  1.0332539014590738, 0.5071337456476289
  1.0389361450923733, 0.5471388978876818
  1.044911497856278, 0.5857948802547399
  1.0511430385858738, 0.6230725684225062
  1.0575948803664124, 0.6589549916162215
  1.0642325542964317, 0.6934362740150675
  1.0710233137464815, 0.726520489186372
  1.077936362448228, 0.7582204737756266
  1.084943012111261, 0.7885566399681238
  1.0920167768670694, 0.8175558189729485
  1.0991334127570234, 0.845250160518976
  1.1062709108203568, 0.8716761065155711
  1.1134094522187117, 0.896873450900018
  1.1205313333747158, 0.9208844924294416
  1.127620868412807, 0.9437532828377022
  1.1346642753651206, 0.9655249693529719
  1.1416495517198535, 0.9862452279931906
  1.1485663440025442, 1.0059597822269548
  1.1554058152340034, 1.0247140003940949
  1.1621605133295811, 1.0425525646077243
  1.1688242428093665, 1.0595192035983518
  1.1753919415842726, 1.0756564820111365
  1.188223971450104, 1.1056064689450926
  1.2006345146926165, 1.1327146373440216
  1.212612911761287, 1.1572680295028432
  1.2241572840060964, 1.1795280277572262
  1.235272352635196, 1.1997308788914158
  1.245967675886558, 1.2180889238656545
  1.2562562469770926, 1.234792234709331
  1.2661533948648618, 1.2500104488099948
  1.2756759339690638, 1.2638946582665762
  1.2848415152847363, 1.2765792617931144
  1.2936681382950788, 1.2881837224414268
  1.3021737898626484, 1.2988141994379099
  1.318292569214076, 1.3175201238605159
  1.3333333333333337, 1.333333333333334
  1.3474211896048045, 1.3467582013661903
  1.3606696652944188, 1.3581960324808218
  1.3731807701378078, 1.3679674935820967
  1.3850457169585462, 1.3763299395433504
  1.3963459435505214, 1.3834907634865314
  1.4071542340922618, 1.3896176848414257
  1.4175358310234483, 1.3948466893103542
  1.4275494828305517, 1.3992881692114159
  1.437248404663158, 1.4030316811324863
  1.4466811464080984, 1.4061496359965266
  1.4558923724238833, 1.408700158840359
  1.4649235619879528, 1.4107292964885563
  1.4738136416787524, 1.4122727063506426
  1.4825995616243057, 1.4133569251801439
  1.5000000000000013, 1.4142135623730967
  1.5174004383756943, 1.4133569251801428
  1.526186358321248, 1.4122727063506437
  1.5350764380120483, 1.4107292964885567
  1.5441076275761165, 1.4087001588403598
  1.5533188535919022, 1.4061496359965262
  1.562751595336841, 1.4030316811324846
  1.572450517169448, 1.3992881692114159
  1.5824641689765526, 1.394846689310353
  1.5928457659077364, 1.3896176848414252
  1.6036540564494766, 1.3834907634865314
  1.6149542830414545, 1.3763299395433521
  1.626819229862194, 1.3679674935820962
  1.6393303347055812, 1.3581960324808229
  1.6525788103951944, 1.346758201366189
  1.6595109955466745, 1.340317103672736
  1.6630600027954385, 1.3368962128403283
  1.6648560135589716, 1.335132936582554
  2.132147586144324, 0.867852413855676
}

\pgfkeys{/pgf/data/new group=twin-splitting-eig}
\pgfdata[store in group=twin-splitting-eig]{
  x, y
  0.9999999999999998, 0.0
  3.0000000000000027, 0.0
  1.5000000000000009, 1.3228756555322971
  1.5000000000000009, -1.3228756555322971
  0.9999999999999998, 0.0
}

\pgfkeys{/stored graphs/nrml-rest/.cd,
  x min/.initial=-1.0,
  x max/.initial=1.0,
  y min/.initial=-1.0,
  y max/.initial=1.0,
  graph code/.code={
    \node (v0) at (-1.0, -0.5) {};
    \node (v1) at (-0.5, -1.0) {};
    \node (v2) at (1.0, -1.0) {};
    \node (v3) at (1.0, 0.5) {};
    \node (v4) at (0.5, 1.0) {};
    \node (v5) at (-1.0, 1.0) {};
    \node (v6) at (0.0, 0.0) {};

    \graph {
      (v0); (v1); (v2); (v3); (v4); (v5); (v6);
      v0 <-> v6;
      v0 -> v2;
      v1 <-> v6;
      v1 -> v2;
      v2 -> { v3, v4 };
      v3 <-> v6;
      v3 -> v5;
      v4 <-> v6;
      v4 -> v5;
      v5 -> { v0, v1 };
    };
  }
}

\pgfkeys{/pgf/data/new group=nrml-rest-rnr}
\pgfdata[store in group=nrml-rest-rnr]{
  x, y
  1.9999999999999996, 0.0
  2.0000000000000027, -1.4142135623730956
  5.414213562373094, 0.0
  2.0000000000000027, 1.4142135623730956
}

\pgfkeys{/pgf/data/new group=nrml-rest-eig}
\pgfdata[store in group=nrml-rest-eig]{
  x, y
  1.9999999999999996, 0.0
  5.414213562373094, 0.0
  2.0000000000000027, 1.4142135623730956
  2.0000000000000027, -1.4142135623730956
  2.585786437626906, 0.0
  2.0000000000000004, 0.0
}

\pgfkeys{/stored graphs/rnrml-non-djoin/.cd,
  x min/.initial=-1.0,
  x max/.initial=1.0,
  y min/.initial=-1.0,
  y max/.initial=1.0,
  graph code/.code={
    \node (v0) at (0.6, 0.0) {};
    \node (v1) at (0.0, 0.6) {};
    \node (v2) at (-0.6, 0.0) {};
    \node (v3) at (0.0, -0.6) {};
    \node (v4) at (1.0, 1.0) {};
    \node (v5) at (-1.0, 1.0) {};
    \node (v6) at (-1.0, -1.0) {};
    \node (v7) at (1.0, -1.0) {};

    \graph {
      (v0); (v1); (v2); (v3); (v4); (v5); (v6); (v7);
      v0 <-> v4;
      v0 -> { v1, v3 };
      v1 <-> { v4, v5 };
      v2 <-> v6;
      v2 -> { v1, v3 };
      v3 <-> { v6, v7 };
      v5 -> { v4, v6 };
      v7 -> { v4, v6 };
    };
  }
}

\pgfkeys{/pgf/data/new group=rnrml-non-djoin-rnr}
\pgfdata[store in group=rnrml-non-djoin-rnr]{
  x, y
  0.5857864376269055, 0.0
  3.000000000000002, -1.0000000000000009
  4.000000000000001, 0.0
  3.000000000000002, 1.0000000000000009
}

\pgfkeys{/pgf/data/new group=rnrml-non-djoin-eig}
\pgfdata[store in group=rnrml-non-djoin-eig]{
  x, y
  0.5857864376269055, 0.0
  2.0000000000000018, 0.0
  3.000000000000002, 1.0000000000000009
  3.000000000000002, -1.0000000000000009
  4.000000000000001, 0.0
  3.414213562373094, 0.0
  4.0, 0.0
}

\pgfkeys{/stored graphs/non-poly-zero-alpha/.cd,
  x min/.initial=-0.694362780359285,
  x max/.initial=1.0,
  y min/.initial=-0.35806098650698454,
  y max/.initial=0.35806098650698454,
  graph code/.code={
    \node (v0) at (-0.07418295952095352, 0.0) {};
    \node (v1) at (-0.694362780359285, 0.35806098650698454) {};
    \node (v2) at (-0.694362780359285, -0.35806098650698454) {};
    \node (v3) at (0.4629085202395232, 0.0) {};
    \node (v4) at (1.0, 0.0) {};

    \graph {
      (v0); (v1); (v2); (v3); (v4);
      v0 -> { v1, v3 };
      v1 -> v2;
      v2 -> v0;
      v3 -> v4;
    };
  }
}

\pgfkeys{/pgf/data/new group=non-poly-zero-alpha-rnr}
\pgfdata[store in group=non-poly-zero-alpha-rnr]{
  x, y
  2.430073525436771, 0.0
  2.4299516751879, -0.01985972979445733
  2.4295864684279067, -0.03969877833574273
  2.4289789341048116, -0.05949657354679038
  2.4281307771443976, -0.07923276035026675
  2.4270443636489722, -0.09888730581736076
  2.4257227005307658, -0.11844060037857812
  2.424169409922862, -0.1378735539195948
  2.422388698792072, -0.1571676856955242
  2.4203853242478495, -0.17630520712761294
  2.4181645550980453, -0.19526909669315973
  2.415732130244858, -0.21404316627766784
  2.41309421454219, -0.2326121185230972
  2.4102573527487183, -0.2509615948727807
  2.4072284222101086, -0.2690782141775516
  2.404014584889258, -0.2869496018847289
  2.400623239337464, -0.3045644099781841
  2.393338516503404, -0.338984085367139
  2.385436396468724, -0.37226719564939387
  2.3769799081355893, -0.4043599859049579
  2.36803169900735, -0.43522427499676747
  2.358652938301142, -0.46483640008991245
  2.348902408107851, -0.49318591511058285
  2.338835786901492, -0.5202741326014908
  2.3285051190173376, -0.5461125893481356
  2.317958455761469, -0.570721503884714
  2.307239648551195, -0.594128280278828
  2.296388271636591, -0.6163660988464661
  2.2854396510966972, -0.6374726216798743
  2.274424977463972, -0.6574888297193563
  2.2633714810431207, -0.676457998879401
  2.2523026513450195, -0.6944248155118805
  2.241238484718601, -0.7114346261417154
  2.230195746978703, -0.727532812723933
  2.219188240418329, -0.7427642823749795
  2.2082270669466335, -0.757173059344096
  2.1973208811482343, -0.7708019666469138
  2.1864761287924868, -0.7836923850485982
  2.175697267738201, -0.7958840777617009
  2.164986969300003, -0.8074150701565375
  2.154346298998662, -0.8183215748486546
  2.143774876242238, -0.8286379536391268
  2.133271012912351, -0.8383967088762015
  2.1228318310908354, -0.8476284978391312
  2.112453360283537, -0.8563621646910503
  2.102130614501288, -0.8646247853932595
  2.0918576494599517, -0.8724417217127923
  2.081627599971581, -0.8798366810880742
  2.071432697322278, -0.8868317796470095
  2.051112699212756, -0.8997032845832389
  2.030816759223388, -0.9112037273752982
  2.010446899916843, -0.9214589728611737
  1.9898841790458652, -0.9305747171894923
  1.968982239432347, -0.9386373573210276
  1.9475571284636475, -0.9457139754747145
  1.9253713991246908, -0.9518509522502479
  1.9021090695243394, -0.9570703392467546
  1.8899445352055546, -0.9593344532221296
  1.8773354464954748, -0.9613623627823138
  1.8641981531115315, -0.9631455139313781
  1.8504310956493146, -0.9646709064768207
  1.8359098236979263, -0.9659197668352602
  1.8204804119746467, -0.9668656959821542
  1.8039507149308907, -0.9674720713065189
  1.7860787094962078, -0.9676883834014045
  1.7665569426170475, -0.967445050618816
  1.7449918515701646, -0.9666460656934188
  1.7208765508439017, -0.965158589769243
  1.6935558180402477, -0.9627983538894445
  1.6621830228547503, -0.9593095851144635
  1.644648285298056, -0.9570361282830789
  1.6256718316557977, -0.9543385015476404
  1.6050732346950205, -0.9511518562398047
  1.582653019872541, -0.9474010672145528
  1.5581936652656332, -0.9429998703994028
  1.5314624501635257, -0.9378504018741867
  1.502216969326804, -0.9318434469882844
  1.4702142782527847, -0.9248598162330313
  1.4352246774343498, -0.9167733656016228
  1.3970509599166938, -0.907456231883395
  1.355553377215366, -0.8967867907466263
  1.3106794667108916, -0.88466057949522
  1.2869917874522296, -0.8780265356829442
  1.2624961698861448, -0.8710038745429904
  1.2372245349927895, -0.8635906630140562
  1.2112195453587509, -0.8557887571530303
  1.1845349199789172, -0.8476041888994339
  1.1572354310524173, -0.8390474666021877
  1.1293965371549342, -0.8301337620925575
  1.1011036242201109, -0.8208829603341488
  1.0724508482165835, -0.8113195536127304
  1.0435395996813535, -0.8014723706522535
  1.0144766379579093, -0.7913741413789182
  0.9853719691191423, -0.7810609093645464
  0.9563365630184533, -0.7705713150731887
  0.9274800189631003, -0.7599457826396026
  0.8989082942569199, -0.7492256498809697
  0.8707216046465556, -0.7384522847637338
  0.8430125912059365, -0.7276662312612145
  0.8158648263090549, -0.7169064236077052
  0.7893517048388157, -0.7062095010142426
  0.7635357388004497, -0.6956092459312864
  0.7384682470503571, -0.6851361590521037
  0.7141894093436186, -0.6748171745449134
  0.6907286369377337, -0.6646755103935738
  0.6681052012429614, -0.6547306418457505
  0.6254017996768044, -0.6354910438356783
  0.586064722947441, -0.6171843602155108
  0.5499789378909969, -0.5998490342853695
  0.5169608070162465, -0.5834857693913862
  0.4867858129393361, -0.568067436985926
  0.4592105150861321, -0.5535479668828397
  0.43398862121280496, -0.5398696837123684
  0.4108818681772194, -0.5269689752732518
  0.3896667386296879, -0.514780426947377
  0.3701380683215735, -0.5032396707436175
  0.3521104700086797, -0.4922852249879152
  0.3354183141300601, -0.4818595804197725
  0.3199148202030469, -0.47190974751353215
  0.3054706524226368, -0.46238743430510115
  0.2793203187941722, -0.44445515137446423
  0.25621883995355244, -0.42776465837020994
  0.2355940078996365, -0.41207827071638303
  0.21700771068464486, -0.39720780797064814
  0.20012276237774976, -0.38300460203164427
  0.18467769688352315, -0.3693510896946296
  0.17046793774795677, -0.35615403716971294
  0.1573318340212989, -0.3433392013300647
  0.14514032126453846, -0.33084717078528314
  0.13378925213045045, -0.3186301422150402
  0.12319368424875947, -0.3066494245124466
  0.11328360313222785, -0.29487350410273844
  0.10400069996148298, -0.28327654133459845
  0.09529592814036085, -0.2718371978731027
  0.08712763783717022, -0.2605377186845247
  0.07946014202804959, -0.24936321042736778
  0.07226260668986607, -0.238301071932808
  0.06550818605016748, -0.22734054295160872
  0.053237319993102376, -0.20568839628813282
  0.04249008233184528, -0.18434556126656376
  0.03314168637103181, -0.1632636316524057
  0.025091892661556602, -0.14240261186213654
  0.018259643840396154, -0.12172804773904916
  0.012579447491066806, -0.10120913844681492
  0.007998894575537825, -0.08081748672268553
  0.0044769372561046985, -0.06052626717562327
  0.001982691936097942, -0.040309668715717006
  0.0004946207418417591, -0.02014251535558226
  -5.551115123125783e-17, -5.024191996501964e-17
  0.0004946207418417814, 0.02014251535558216
  0.0019826919360979318, 0.04030966871571688
  0.004476937256104714, 0.060526267175623094
  0.007998894575537846, 0.0808174867226855
  0.012579447491066776, 0.10120913844681481
  0.018259643840396106, 0.12172804773904906
  0.025091892661556467, 0.14240261186213615
  0.03314168637103185, 0.16326363165240582
  0.04249008233184518, 0.18434556126656365
  0.053237319993102244, 0.20568839628813262
  0.06550818605016741, 0.22734054295160852
  0.07226260668986596, 0.23830107193280797
  0.07946014202804941, 0.24936321042736742
  0.08712763783717001, 0.2605377186845242
  0.09529592814036081, 0.2718371978731026
  0.10400069996148283, 0.283276541334598
  0.11328360313222767, 0.2948735041027384
  0.12319368424875951, 0.30664942451244653
  0.1337892521304503, 0.31863014221504
  0.14514032126453824, 0.330847170785283
  0.15733183402129877, 0.34333920133006457
  0.17046793774795635, 0.3561540371697124
  0.18467769688352298, 0.3693510896946295
  0.20012276237774942, 0.38300460203164405
  0.2170077106846442, 0.39720780797064753
  0.2355940078996356, 0.4120782707163823
  0.2562188399535522, 0.4277646583702095
  0.2793203187941723, 0.4444551513744643
  0.3054706524226367, 0.4623874343051012
  0.3199148202030457, 0.4719097475135313
  0.3354183141300601, 0.4818595804197726
  0.35211047000867846, 0.49228522498791444
  0.37013806832157264, 0.5032396707436169
  0.3896667386296864, 0.5147804269473759
  0.4108818681772187, 0.5269689752732515
  0.4339886212128049, 0.5398696837123683
  0.45921051508613087, 0.5535479668828389
  0.48678581293933604, 0.5680674369859259
  0.5169608070162461, 0.5834857693913861
  0.5499789378909969, 0.5998490342853696
  0.5860647229474392, 0.6171843602155098
  0.625401799676804, 0.6354910438356781
  0.668105201242959, 0.6547306418457495
  0.6907286369377337, 0.6646755103935738
  0.7141894093436182, 0.6748171745449132
  0.7384682470503557, 0.6851361590521032
  0.7635357388004461, 0.6956092459312846
  0.7893517048388157, 0.7062095010142426
  0.8158648263090528, 0.7169064236077041
  0.8430125912059355, 0.7276662312612141
  0.8707216046465582, 0.7384522847637351
  0.8989082942569199, 0.7492256498809697
  0.9274800189630998, 0.7599457826396023
  0.9563365630184512, 0.7705713150731879
  0.9853719691191423, 0.7810609093645464
  1.0144766379579078, 0.7913741413789173
  1.0435395996813526, 0.8014723706522537
  1.072450848216584, 0.8113195536127306
  1.1011036242201109, 0.8208829603341488
  1.129396537154935, 0.8301337620925581
  1.1572354310524162, 0.8390474666021872
  1.1845349199789161, 0.8476041888994332
  1.2112195453587484, 0.8557887571530294
  1.237224534992786, 0.8635906630140546
  1.2624961698861423, 0.87100387454299
  1.2869917874522296, 0.8780265356829442
  1.3106794667108925, 0.8846605794952205
  1.3555533772153678, 0.8967867907466269
  1.3970509599166936, 0.9074562318833952
  1.4352246774343502, 0.9167733656016229
  1.470214278252784, 0.9248598162330317
  1.502216969326804, 0.9318434469882841
  1.5314624501635254, 0.9378504018741869
  1.5581936652656332, 0.9429998703994029
  1.5826530198725395, 0.9474010672145521
  1.6050732346950194, 0.9511518562398042
  1.6256718316557972, 0.9543385015476408
  1.644648285298056, 0.957036128283079
  1.6621830228547505, 0.9593095851144633
  1.693555818040249, 0.9627983538894453
  1.7208765508439026, 0.9651585897692434
  1.7449918515701635, 0.9666460656934182
  1.7665569426170453, 0.967445050618815
  1.7860787094962083, 0.9676883834014047
  1.8039507149308907, 0.967472071306519
  1.8204804119746465, 0.9668656959821542
  1.8359098236979263, 0.9659197668352604
  1.8504310956493137, 0.9646709064768201
  1.8641981531115326, 0.9631455139313788
  1.877335446495475, 0.961362362782314
  1.889944535205555, 0.9593344532221296
  1.9021090695243394, 0.9570703392467547
  1.9253713991246897, 0.9518509522502475
  1.9475571284636464, 0.9457139754747139
  1.9689822394323468, 0.9386373573210278
  1.9898841790458652, 0.9305747171894923
  2.010446899916843, 0.9214589728611737
  2.030816759223387, 0.911203727375298
  2.0511126992127564, 0.8997032845832391
  2.071432697322278, 0.8868317796470098
  2.081627599971581, 0.8798366810880744
  2.0918576494599526, 0.872441721712793
  2.102130614501289, 0.8646247853932598
  2.112453360283537, 0.8563621646910503
  2.122831831090834, 0.8476284978391307
  2.133271012912349, 0.838396708876201
  2.1437748762422397, 0.8286379536391277
  2.1543462989986635, 0.8183215748486552
  2.1649869693000032, 0.8074150701565377
  2.175697267738201, 0.7958840777617012
  2.186476128792487, 0.7836923850485986
  2.197320881148235, 0.7708019666469146
  2.2082270669466326, 0.7571730593440962
  2.2191882404183283, 0.742764282374979
  2.230195746978704, 0.7275328127239333
  2.241238484718601, 0.7114346261417154
  2.2523026513450217, 0.6944248155118808
  2.263371481043124, 0.6764579988794023
  2.2744249774639735, 0.657488829719357
  2.2854396510966963, 0.6374726216798742
  2.2963882716365918, 0.6163660988464671
  2.307239648551196, 0.594128280278829
  2.3179584557614685, 0.5707215038847149
  2.3285051190173367, 0.5461125893481363
  2.338835786901493, 0.520274132601491
  2.3489024081078496, 0.49318591511058285
  2.358652938301142, 0.46483640008991245
  2.36803169900735, 0.43522427499676797
  2.3769799081355893, 0.4043599859049584
  2.385436396468726, 0.3722671956493947
  2.393338516503406, 0.3389840853671399
  2.400623239337464, 0.3045644099781848
  2.404014584889258, 0.28694960188472896
  2.4072284222101077, 0.2690782141775526
  2.4102573527487183, 0.2509615948727809
  2.4130942145421868, 0.23261211852309796
  2.4157321302448582, 0.21404316627766817
  2.418164555098045, 0.1952690966931596
  2.4203853242478486, 0.17630520712761336
  2.422388698792072, 0.15716768569552403
  2.4241694099228632, 0.13787355391959558
  2.4257227005307658, 0.11844060037857812
  2.4270443636489736, 0.09888730581736153
  2.4281307771443976, 0.07923276035026691
  2.428978934104812, 0.05949657354679121
  2.4295864684279067, 0.03969877833574299
  2.429951675187899, 0.019859729794458363
}

\pgfkeys{/pgf/data/new group=non-poly-zero-alpha-eig}
\pgfdata[store in group=non-poly-zero-alpha-eig]{
  x, y
  1.8774388331233465, 0.7448617666197445
  1.8774388331233465, -0.7448617666197445
  0.24512233375330733, 0.0
  1.0, 0.0
}

\tikzset{/my/graphic options,defaults/.style={wide,node outline size=0.8,edge size=0.9,sep=0.7cm,graph width=5cm,rnr width=4cm,graph height=3.5cm,rnr height=4cm},tall preset/.style={/my/graphic options/defaults/.append style={tall,graph width=3.3cm,rnr width=3.6cm,graph height=2.6cm,rnr height=3.2cm}}}

\ifpdf
\hypersetup{
  pdftitle={On Digraphs with Polygonal Restricted Numerical Range},
  pdfauthor={Thomas R. Cameron, H. Tracy Hall, Ben Small, and Alexander Wiedemann}
}
\fi

\begin{document}

\maketitle

\begin{abstract}
In 2020, Cameron et al. introduced the restricted numerical range of a digraph (directed graph) as a tool for characterizing digraphs and studying their algebraic connectivity. 
In particular, digraphs with a restricted numerical range of a single point, a horizontal line segment, and a vertical line segment were characterized as $k$-imploding stars, directed joins of bidirectional digraphs, and regular tournaments, respectively. 
In this article, we extend these results by investigating digraphs whose restricted numerical range is a convex polygon in the complex plane. 
We provide computational methods for identifying these polygonal digraphs and show that these digraphs can be broken into three disjoint classes: normal, restricted-normal, and pseudo-normal digraphs, all of which are closed under the digraph complement. 
We prove sufficient conditions for normal digraphs and show that the directed join of two normal digraphs results in a restricted-normal digraph.
Also, we prove that directed joins are the only restricted-normal digraphs when the order is square-free or twice a square-free number.
Finally, we provide methods to construct restricted-normal digraphs that are not directed joins for all orders that are neither square-free nor twice a square-free number.
\end{abstract}

\section{Introduction}	\label{sec:intro}
Let $\mathbb{G}$ denote the set of finite simple unweighted digraphs.
For each $\Gamma\in\mathbb{G}$, we have $\Gamma=(V,E)$, where $V$ is the vertex set, $E\subseteq V\times V$ is the edge set, and $(i,j)\in E$ if and only if $i\neq j$ and there is an edge from vertex $i$ to vertex $j$.
The \emph{order} of $\Gamma$ is equal to $\abs{V}$; when $\abs{V}=0$, we refer to $\Gamma$ as the \emph{null digraph}. 

Given $\Gamma\in\mathbb{G}$ of order $n$, we denote the \emph{out-degree} of the vertex $i\in V$ by $d^{+}(i)$, which is equal to the number of edges of the form $(i,j)\in E$. 
Similarly, we denote the \emph{in-degree} of the vertex $i\in V$ by $d^{-}(i)$, which is equal to the number of edges of the form $(j,i)\in E$.
After indexing the vertex set as $V=\{1,2,\ldots,n\}$, we define the \emph{Laplacian} matrix of $\Gamma$ by $L=[l_{ij}]_{i,j=1}^{n}$, where
\[
l_{ij} = \begin{cases} d^{+}(i) & \text{if $i=j$} \\ -a_{ij} & \text{if $i\neq j$} \end{cases}
\]
and $A = [a_{ij}]_{i,j=1}^{n}$ is the \emph{adjacency matrix} of $\Gamma$, that is, $a_{ij} = 1$ if $(i,j)\in E$ and $a_{ij} = 0$ otherwise. We use subscript notation to indicate the particular digraph when it is unclear from the surrounding context, for example, $d_{\Gamma}^{+}(i)$ denotes the out-degree of vertex $i$ in digraph $\Gamma$.

In general, the \emph{numerical range} (or \emph{field of values}) of a complex matrix $A\in\mathbb{C}^{n\times n}$ is defined as follows~\cite{Kippenhahn1951,Zachlin2008}:
\[
W(A) = \left\{\mathbf{x}^{*}A\mathbf{x}\colon~\mathbf{x}\in\mathbb{C}^{n},~\norm{\mathbf{x}}=1\right\},
\]
where $\norm{\cdot}$ denotes the Euclidean norm on complex vectors.
For our purposes, we are interested in the \emph{restricted numerical range} of the Laplacian matrix, which we define by
\[
W_{r}(L) = \left\{ \mathbf{x}^{*}L\mathbf{x}\colon~\mathbf{x}\in\mathbb{C}^{n},~\mathbf{x}\perp\mathbf{e},~\norm{\mathbf{x}}=1\right\},
\]
where $\mathbf{e}$ is the all ones vector of dimension $n$.
Clearly $W_{r}(L)=\emptyset$ when $n=1$.
If the dimension is unclear from context, we use the notation $\mathbf{e}^{n}$. 
When convenient, we may refer to $W_{r}(L)$ as the restricted numerical range of a digraph, and mix the notation $W_{r}(L)$ with $W_{r}(\Gamma)$.

The definition of the restricted numerical range is motivated by its close connection to the algebraic connectivity for digraphs as defined in~\cite{Wu2005}.
Indeed, let $\Gamma\in\mathbb{G}$ have order $n$ and let $L$ be the Laplacian matrix of $\Gamma$. 
Then, the \emph{algebraic connectivity} of $\Gamma$ is defined by
\[
\alpha(\Gamma) = \min_{\mathbf{x}\in\mathcal{S}}\mathbf{x}^{T}L\mathbf{x},
\]
where
\[
\mathcal{S} = \left\{\mathbf{x}\in\mathbb{R}^{n}\colon~\mathbf{x}\perp\mathbf{e},~\norm{\mathbf{x}}=1\right\}.
\]
Another related and useful quantity is
\[
\beta(\Gamma) = \max_{\mathbf{x}\in\mathcal{S}}\mathbf{x}^{T}L\mathbf{x}.
\]
We summarize this connection and other basic properties below.
Note thate we define a \emph{restrictor matrix} of order $n$ as an $n\times(n-1)$ orthonormal matrix whose columns are orthogonal to $\mathbf{e}$.
If $n=1$, then the restrictor matrix is an \emph{empty matrix}, that is, a matrix where the number of columns is zero.
Also, we use $\conj{\Gamma}$ to denote the \emph{complement digraph} of $\Gamma$, that is, the digraph whose edge set consists exactly of those directed edges not in $\Gamma$.
\begin{proposition}\label{prop:basic-rnr}
Let $\Gamma\in\mathbb{G}$ have order $n$ and let $L$ be the Laplacian matrix of $\Gamma$.
Also, let $Q$ be a restrictor matrix of order $n$.
Then, the following properties hold.
\begin{enumerate}[(i)]
\item The restricted numerical range satisfies $W_{r}(L)=W(Q^{*}LQ)$.
\item	The set $W_{r}(L)$ is invariant under re-ordering of the vertices of $\Gamma$.
\item	The eigenvalues of $L$ are contained in $W_{r}(L)$, except (possibly) for the zero eigenvalue associated with the eigenvector $\mathbf{e}$.
\item	The minimum real part of $W_{r}(L)$ is equal to $\alpha(\Gamma)$ and the maximum real part of $W_{r}(L)$ is equal to $\beta(\Gamma)$.
\item	Let $\conj{L}$ denote the Laplacian matrix of the complement digraph. Then, $W_{r}(\conj{L}) = n - W_{r}(L)$.
\end{enumerate}
\end{proposition}
\begin{proof}
For (i)--(iv), see~\cite[Proposition 2.1]{Cameron2020_RNR}.
For (v), note that $\conj{L} = \left(nI - \mathbf{e}\mathbf{e}^{T}\right) - L$.
Therefore, 
\begin{align*}
W_{r}(\conj{L}) &= W(Q^{*}\conj{L}Q) \\
&= W(nI - Q^{*}LQ) \\
&= n - W(Q^{*}LQ),
\end{align*}
by~\cite[Theorem 4]{Kippenhahn1951,Zachlin2008}.
\end{proof}

Next, we summarize the known characterizations of digraphs using the restricted numerical range.
Note that the \emph{directed join} of digraphs $\Gamma = (V,E)$ and $\Gamma'=(V',E')$, where $V\cap V'=\emptyset$, is defined by
\[
\Gamma\djoin\Gamma' = \left(V\sqcup V',E\sqcup E'\sqcup\left\{(i,j)\colon~i\in V,j\in V'\right\}\right),
\]
and $E_{n}$ and $K_{n}$ denote the empty and complete digraph of order $n$, respectively.
\begin{theorem}\label{thm:rnr_char}
Let $\Gamma\in\mathbb{G}$ have order $n$  and let $L$ be the Laplacian matrix of $\Gamma$.
Then, the following characterizations hold:
\begin{enumerate}[(i)]
\item $\Gamma$ is a dicycle (directed cycle) of order $n$ if and only if $W_{r}(L)$ is a complex polygon with vertices
\[
\left\{1-e^{\iu 2\pi j/n}\colon j=1,\ldots,n-1\right\},
\]
where $\iu$ denotes the imaginary unit.
\item $\Gamma$ is a $k$-imploding star, that is, $\Gamma = E_{n-k}\djoin K_{k}$ for some $k\in\{0,1,\ldots,n\}$, if and only if $W_{r}(L)$ is a single point.
Moreover, the numerical value of this point is $k$. 
\item $\Gamma$ is a regular tournament if and only if $W_{r}(L)$ is a vertical line segment. Moreover, $n$ is odd and this vertical line segment has real part $n/2$.
\item $\Gamma$ is a directed join of two bidirectional digraphs, where one may be the null digraph, if and only if $W_{r}(L)$ is a horizontal line segment. Moreover, this line segment lies on the non-negative portion of the real axis. 
\end{enumerate}
\end{theorem}
\begin{proof}
For (i), see~\cite[Theorem 2.6]{Cameron2020_RNR}; for (ii), see~\cite[Theorem 3.2]{Cameron2020_RNR}; for (iii), see~\cite[Theorem 3.3 and Corollary 3.4]{Cameron2020_RNR}; for (iv), see~\cite[Theorem 4.2 and Theorem 4.3]{Cameron2020_RNR}.
\end{proof}

Note that parts (ii)--(iv) of Theorem~\ref{thm:rnr_char} completely describe all digraphs with a restricted numerical range as a degenerate polygon, that is, a point or a line segment in the complex plane. 
In this article, we extend these results to include digraphs whose restricted numerical range is a non-degenerate convex polygon in the complex plane.
Throughout, we refer to digraphs whose restricted numerical range is a degenerate or non-degenerate convex polygon in the complex plane as \emph{polygonal}.

Partial motivation for studying polygonal digraphs comes from their application to synchronization theory. 
Indeed, consider the parameter $\mu(L)$, which is defined as the supremum of the set of real numbers $\mu$ such that $U(L-\mu I) + (L^{T}-\mu I)U$ is positive semi-definite for some irreducible symmetric zero sum matrix $U$ with non-positive off-diagonal entries. 
In~\cite{Wu2006}, it is shown that $\mu(L)$ is a measure of how well the topology of the associated coupled network is amenable to synchronization.
In particular, the larger $\mu(L)$ is, the easier it is to synchronize the network. 
The following proposition extends the result in~\cite[Theorem 4]{Wu2006}.
Note that $\sigma(\cdot)$ denotes the multi-set of eigenvalues for a given square complex matrix.
\begin{proposition}\label{prop:poly_mu}
Let $\Gamma\in\mathbb{G}$ have order $n$ and let $L$ be the Laplacian matrix of $\Gamma$.
If $\Gamma$ is polygonal, then $\alpha(\Gamma) = \mu(L)$.
\end{proposition}
\begin{proof}
By~\cite[Theorem 2 and Theorem 3]{Wu2006}, it follows that
\[
\alpha(\Gamma)\leq \mu(L)\leq\min_{\lambda\in\sigma(Q^{*}LQ)}\re{\lambda},
\]
where $Q$ is a restrictor matrix of order $n$.
Since $\Gamma$ is polygonal, $W_{r}(L)$ is a convex polygon whose vertices, by Proposition~\ref{prop:basic-rnr} (i) and~\cite[Theorem 13]{Kippenhahn1951,Zachlin2008}, are eigenvalues of $Q^{*}LQ$.
Therefore,
\[
\alpha(\Gamma) = \min_{\lambda\in\sigma(Q^{*}LQ)}\re{\lambda}.
\]
\end{proof}

In~\cite{Asadi2016}, the value
\[
\min_{\lambda\in\sigma(Q^{*}LQ)}\re{\lambda}
\]
is called the generalized algebraic connectivity of $\Gamma$ and it is shown that this value reflects the expected asymptotic convergence rate of cooperative consensus-based algorithms in an asymmetric network represented by $\Gamma$.
Hence, our investigation of polygonal digraphs will help identify networks for which the algebraic connectivity $\alpha(\Gamma)$, the generalized algebraic connectivity from~\cite{Asadi2016}, and the synchronization parameter $\mu(L)$ from~\cite{Wu2006} are equal.

In addition, our study of polygonal digraphs will help identify digraphs for which the algebraic connectivity satisfies $\alpha(\Gamma)=0$ if and only if $\Gamma$ has multiple terminal strongly connected components, see Proposition~\ref{prop:poly_zero_alpha}. 
Ideally, this result would hold for all digraphs since $\alpha(\Gamma)$ is a generalization of Fiedler's algebraic connectivity for undirected graphs~\cite{Fiedler1973} and the number of terminal strongly connected components in a digraph is equal to the algebraic (and geometric) multiplicity of the zero eigenvalue of the corresponding Laplacian matrix~\cite{Mirzaev2013}.
However, as illustrated by Figure 1, this result is not in general true for all digraphs. 
Indeed, Figure~\ref{fig:non-poly-zero-alpha} shows a digraph that has only one terminal strongly connected component, yet its algebraic connectivity is zero. 
Note that the digraph is shown on the left with the restricted numerical range on the right, and the eigenvalues of $Q^{*}LQ$ are displayed using the star symbol.
The digraph vertices are not labeled since, by Proposition~\ref{prop:basic-rnr}, the restricted numerical range is invariant under re-ordering of the vertices.

Note that given any $\Gamma=(V,E)\in\mathbb{G}$ we can partition the vertices as $V=V_{1}\sqcup V_{2}\sqcup\cdots\sqcup V_{k}$, where for each $i=1,\ldots,k$, the subdigraph induced by $V_{i}$ is \emph{strongly connected}, that is, there exists a directed path between all pairs of vertices in $V_{i}$.
Also, we say that the subdigraph induced by $V_{i}$ is \emph{terminal} provided that for all $u\in V_{i}$ and $v\in V_{j}$, where $i\neq j$, the edge $(u,v)$ does not exist in $E$.
\begin{proposition}\label{prop:poly_zero_alpha}
Let $\Gamma\in\mathbb{G}$ have order $n$ be polygonal.
Then, $\alpha(\Gamma)=0$ if and only if $\Gamma$ has multiple terminal strongly connected components.
\end{proposition}
\begin{proof}
As noted in the proof of Proposition~\ref{prop:poly_mu}, since $\Gamma$ is polygonal, $W_{r}(L)$ is a convex polygon whose vertices are eigenvalues of $Q^{*}LQ$, which are also eigenvalues of $L$.
Moreover, by Ger{\v s}gorin's circle theorem (see for example~\cite{Horn2013}), the eigenvalues of $L$ are contained in the disk centered at $\max_{i\in V}d^{+}(i)$ with radius $\max_{i\in V}d^{+}(i)$.
Hence, $L$ cannot have purely imaginary eigenvalues, and it follows that $0$ is the minimum real part of $W_{r}(L)$ if and only if the origin of the complex plane is a corner of $W_{r}(L)$. 
Therefore, Proposition~\ref{prop:basic-rnr} (iv) and~\cite[Theorem 13]{Kippenhahn1951,Zachlin2008} imply that $\alpha(\Gamma)=0$ if and only if $0$ is an eigenvalue of $Q^{*}LQ$, that is, $0$ is an eigenvalue of $L$ with multiplicity greater than $1$.
\end{proof}

\begin{figure}[ht]
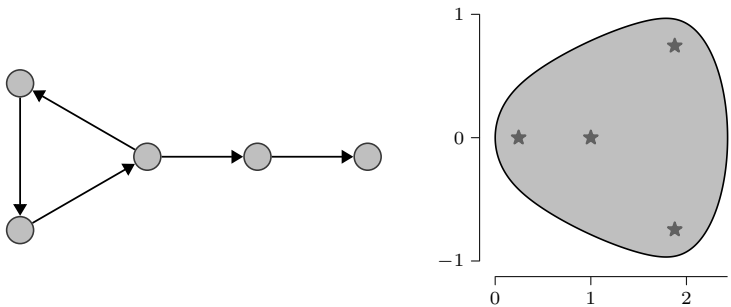

\centering
\drawFigure{non-poly-zero-alpha}{y min=-1,y max=1,x min=-0.000001}
\caption{A digraph with zero algebraic connectivity and one isolated strongly connected component.}
\label{fig:non-poly-zero-alpha}
\end{figure}

Before proceeding, note that our definitions for algebraic connectivity, restricted numerical range, and polygonal digraphs, along with our results from Proposition~\ref{prop:basic-rnr} (i)--(iv), Proposition~\ref{prop:poly_mu}, and Proposition~\ref{prop:poly_zero_alpha} hold for finite simple digraphs with positive weights. 
Although, in this article, we focus on unweighted digraphs. 
\section{Computing Polygonal Digraphs}\label{sec:comp}
Let $\Gamma\in\mathbb{G}$ have order $n$ and let $L$ be the Laplacian matrix of $\Gamma$.
By Proposition~\ref{prop:basic-rnr}, the restricted numerical range of $\Gamma$ can be computed as the numerical range of $Q^{*}LQ$, where $Q$ is a restrictor matrix of order $n$.
The computation of the numerical range of a general complex matrix $B$ is well-established in the literature and relies on the fact that since $W(B)$ is convex, as proved by Toeplitz and Hausdorff~\cite{Hausdorff1919,Toeplitz1918}, the boundary of $W(B)$ can be approximated by convex polygons determined by supporting lines of $W(B)$~\cite{Johnson1978,Kippenhahn1951,Zachlin2008}.
Note that a \emph{supporting line} of a curve is a line that contains a point on the curve but does not separate any two points on the curve.

\subsection{Supporting Lines Check}
Using Johnson's algorithm~\cite{Johnson1978}, we compute the restricted numerical range of a digraph $\Gamma\in\mathbb{G}$ by setting $B = Q^{*}LQ$, where $L$ is the Laplacian matrix of $\Gamma$ and $Q$ is a restrictor matrix of order $n$.
Once $W_{r}(L)$ is computed we can determine if $\Gamma$ is polygonal by checking whether $W_{r}(L)$ is equal to the convex hull of the eigenvalues of $B$, which we denote by $\conv{\sigma(B)}$.
In general, given a finite set $S$ contained in a real or complex vector space, the \emph{convex hull} $\conv{S}$ is the set of all convex combinations of elements from $S$. 
Alternatively, $\conv{S}$ can be characterized as the intersection of all convex sets containing $S$, so it is the smallest closed convex set containing $S$~\cite{Horn1991}.

If we are only interested in whether $W(B)=\conv{\sigma(B)}$, for a given $B\in\mathbb{C}^{n\times n}$, then we need not compute the numerical range $W(B)$.
Indeed, since $\conv{\sigma(B)}\subseteq W(B)$, it suffices to check that every edge of $\conv{\sigma(B)}$ is a supporting line of $W(B)$.
To this end, let $v_{1},\ldots,v_{k}$ denote the vertices of $\conv{\sigma(A)}$ ordered in the counter-clockwise direction.
Then, for each $j\in\{1,\ldots,k\}$, define the direction
\[
d_{j} = \frac{v_{j+1}-v_{j}}{\abs{v_{j+1}-v_{j}}}
\]
and note that the $j$th edge of $\conv{\sigma(B)}$ can be written as
\[
v_{j}v_{j+1} = \left\{v_{j} + td_{j}\colon 0\leq t\leq \abs{v_{j+1}-v_{j}} \right\},
\]
where it is assumed that $v_{k+1}=v_{1}$.

For each $j\in\{1,\ldots,k\}$, let $\theta_{j}\in[0,2\pi)$ be such that 
\begin{equation}\label{eq:theta}
e^{\iu\theta_{j}} = \iu\cdot\conj{d_{j}},
\end{equation}
and define
\[
\hat{v}_{j} = e^{\iu\theta_{j}}v_{j}
\]
as the vertices of the rotated convex hull $\conv{\sigma\left(e^{\iu\theta_{j}}B\right)}$.
Based on how $\theta_{j}$ is defined, it is clear that the $j$th edge of this rotated convex hull,
\[
\hat{v}_{j}\hat{v}_{j+1} = \left\{e^{\iu\theta_{j}}v_{j} + te^{\iu\theta_{j}}d_{j}\colon 0\leq t\leq\abs{v_{j+1}-v_{j}}\right\},
\]
can be formed from the edge $v_{j}v_{j+1}$ by first rotating it to the real axis (pointing in the positive direction) and then rotating $90$ degrees in the counter-clockwise direction.
Hence, the edge $\hat{v}_{j}\hat{v}_{j+1}$ is a vertical line segment whose real part is equal to the maximum real part of $\conv{\sigma\left(e^{\iu\theta_{j}}B\right)}$.
If the edge $\hat{v}_{j}\hat{v}_{j+1}$ is also a supporting line of $W(e^{\iu\theta_{j}}B)$, then it follows that $\re{e^{\iu\theta_{j}}v_{j}}$ must equal the maximum real part of $W(e^{\iu\theta_{j}}B)$, which is well-known
(see for example~\cite{Johnson1978}) to equal the maximum real eigenvalue of the Hermitian part of the matrix $e^{\iu\theta_{j}}B$.

Since $W(B)=\conv{\sigma(B)}$ if and only if $W(e^{\iu\theta}B)=\conv{\sigma\left(e^{\iu\theta}B\right)}$, for all $\theta\in[0,2\pi)$, the above discussion leads to an efficient algorithm for determining if a digraph $\Gamma\in\mathbb{G}$ is polygonal.
Indeed, see Algorithm~\ref{alg:poly}, noting that $H(B) = \left(B+B^{*}\right)/2$ is the \emph{Hermitian part} of the matrix $B$.
\begin{algorithm}
\caption{Algorithm to determine if $\Gamma$ is a polygonal digraph.}
\label{alg:poly}
\begin{algorithmic}
\STATE{I. Let $B=Q^{*}LQ$, where $Q$ is a restrictor matrix of order $n$.}
\STATE{II. Let $v_{1},\ldots,v_{k},v_{k+1}$, where $v_{k+1}=v_{1}$, denote the vertices of $\conv{\sigma(B)}$.}
\STATE{III.} Let $\epsilon$ be some pre-determined tolerance. 
\STATE{IV. If $k=1$, then there is only one real eigenvalue in the convex hull. Perform the following check:}
\IF{$\lambda_{\textrm{max}}\left(H(B)\right) > \lambda_{\textrm{min}}\left(H(B)\right) + \epsilon$}
	\STATE{Return False}
\ELSE
	\STATE{Return True}
\ENDIF
\STATE{V. If $k>1$, then perform the following checks:}
\FOR{$j=1,\ldots,k$}
	\STATE{Define $e^{\iu\theta_{j}}$ as in~\eqref{eq:theta}}
	\IF{$\lambda_{\textrm{max}}\left(H(e^{\iu\theta_{j}}B)\right) > \re{e^{\iu\theta_{j}}v_{j}} + \epsilon$}
		\STATE{Return False}
	\ENDIF
\ENDFOR
\STATE{Return True}
\end{algorithmic}
\end{algorithm}
\subsection{Survey of Polygonal Digraphs}
Using the {\tt Nauty} software~\cite{McKay2013}, we are able to generate all non-isomorphic digraphs of order $n$.
For each digraph $\Gamma$, we can use Algorithm~\ref{alg:poly} to determine if $\Gamma$ is polygonal. 
We can further speed up this process by noting that Proposition~\ref{prop:basic-rnr}(v) implies that $W_{r}(L)$ is a complex polygon if and only if $W_{r}(\conj{L})$ is a complex polygon.
Hence, to generate the set of polygonal digraphs, we only need to search through (about) half of the non-isomorphic digraphs of order $n$ and then add complements as necessary.
Code that implements this process in Python and C is available at~\url{https://github.com/trcameron/polygonal-digraphs}. 
After running this code, we identified $3,~9,~29,~97,~395,~2185,$ and $18930$ polygonal digraphs on $2,~3,~4,~5,~6,~7,$ and $8$ vertices, respectively.

Furthermore, the polygonal digraphs can be split into three classes. 
In \emph{class 1}, the Laplacian matrix is normal; we refer to these as \emph{normal digraphs}.
It is worth noting that~\cite[Lemma 2.5]{Cameron2020_RNR} implies that if $L$ is normal then $Q^{*}LQ$ is normal for any restrictor matrix $Q$ of order $n$.
In \emph{class 2}, the Laplacian matrix is not normal but $Q^{*}LQ$ is normal; we refer to these as \emph{restricted-normal digraphs}.
Finally, in \emph{class 3}, neither the Laplacian matrix nor $Q^{*}LQ$ is normal but $W_{r}(L)$ is still a complex polygon; we refer to these as \emph{pseudo-normal digraphs}.
Figure~\ref{fig:classes} provides an example from each class.
\begin{figure}[ht]
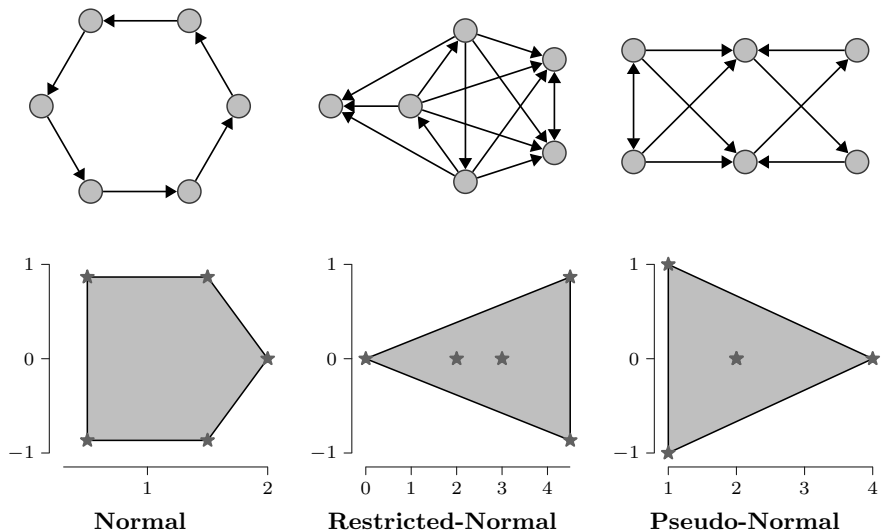

\centering
\tikzset{/my/graphic options/tall preset}
\begin{tabular}{ccc}
\drawFigure{case1}{x min=0.3,y min=-1,y max=1}
&
\drawFigure{case2}{y min=-1,y max=1}
&
\drawFigure{case3}{}
\\
\textbf{\small{Normal}}
&
\textbf{\small{Restricted-Normal}}
&
\textbf{\small{Pseudo-Normal}}
\end{tabular}
\caption{Examples from each class of polygonal digraphs.}
\label{fig:classes}
\end{figure}

In Table~\ref{tab:classes}, the number of polygonal digraphs in each class is shown for orders $n=2,\ldots,8$.
Not surprisingly, there are no pseudo-normal digraphs until $n\geq 6$. 
Indeed, if $W(B)=\conv{\sigma(B)}$ for some $n\times n$ complex matrix $B$, where $n\leq 4$, then it is known that $B$ must be normal~\cite{Johnson1976}.
Since $Q^{*}LQ$ is a square matrix of size one less than the size of $L$, the above observation follows immediately. 
\begin{table}[ht]
\centering
\begin{tabular}{c|c|c|c}
n & Normal Digraphs & Restricted-Normal Digraphs & Pseudo-Normal Digraphs \\
\hline
2 & 2 & 1 & 0 \\
3 & 5 & 4 & 0 \\
4 & 15 & 14 & 0 \\
5 & 47 & 50 & 0 \\
6 & 214 & 179 & 2 \\
7 & 1375 & 766 & 44 \\
8 & 14182 & 4349 & 399 \\
\hline
\end{tabular}
\caption{Number of polygonal digraphs in each class.}
\label{tab:classes}
\end{table}

We conclude this section with Proposition~\ref{prop:class-comp}, which shows that not only are polygonal digraphs, in general, closed under the digraph complement, but the individual classes are also closed under the digraph complement. 
\begin{proposition}\label{prop:class-comp}
Let normal, restricted-normal, and pseudo-normal digraphs be referred to as class $1$, $2$, and $3$, respectively.
Also, let $\Gamma\in\mathbb{G}$ have order $n$.
Then, $\Gamma$ is in class $c\in\{1,2,3\}$ if and only if its complement $\conj{\Gamma}$ is in class $c$.
\end{proposition}
\begin{proof}
Let $\Gamma$ and $\conj{\Gamma}$ be complementary digraphs of order $n$ with Laplacian matrices $L$ and $\conj{L}$, respectively. 
Note that
\[
\conj{L} = \left(n I - \mathbf{e}\mathbf{e}^{T}\right)-L.
\]
Clearly $\mathbf{e}$ is an eigenvector of both $L$ and $\conj{L}$ corresponding to the zero eigenvalue.
Furthermore, for any $\mathbf{v}$ that is orthogonal to $\mathbf{e}$, the eigenvalue-eigenvector equation $L\mathbf{v}=\lambda\mathbf{v}$ holds if and only if $\conj{L}\mathbf{v} = \left(n-\lambda\right)\mathbf{v}$ holds. 
Therefore, by Grone~\cite[Condition 14]{Grone1987}, $L$ is normal if and only if $\conj{L}$ is normal, which implies that $\Gamma$ is normal if and only if $\conj{\Gamma}$ is normal.

Now, let $Q$ be a restrictor matrix of order $n$.
Then, 
\[
Q^{*}\conj{L}Q = n I - Q^{*}LQ.
\]
Hence, $Q^{*}LQ$ is normal if and only if $Q^{*}\conj{L}Q$ is normal, which implies that $\Gamma$ is restricted-normal if and only if $\conj{\Gamma}$ is restricted-normal. 

Finally, by Proposition~\ref{prop:basic-rnr}(v), it follows that $W_{r}(L)$ is a complex polygon if and only if $W_{r}\left(\conj{L}\right)$ is a complex polygon.
Therefore, $\Gamma$ is pseudo-normal if and only if $\conj{\Gamma}$ is pseudo-normal. 
\end{proof}

\section{Analyzing Polygonal Digraphs}\label{sec:analysis-poly}
In this section, we analyze the structure of polygonal digraphs from the normal and restricted-normal classes. 
In particular, we prove sufficient conditions for normal digraphs and show that the directed join of two normal digraphs results in a restricted-normal digraph.
Also, we prove that directed joins are the only restricted-normal digraphs when the order is square-free or twice a square-free number.
Finally, we provide a construction for restricted-normal digraphs that are not directed joins for all orders that are neither square-free nor twice a square-free number. 

\subsection{Normal Digraphs}\label{subsec:analysis-nrml}
It is easy enough to construct a normal digraph.
For instance, any dicycle will suffice as illustrated in Figure~\ref{fig:classes}; more generally, any digraph whose Laplacian is a circulant matrix, possibly after re-ordering the vertices, will also suffice. 
Consider, for example, the digraph in Figure~\ref{fig:circ-normal} whose plane embedding has $60^{\circ}$ rotational symmetry.
Given a certain ordering of the vertices, the Laplacian matrix $L$ can be written as a circulant matrix and is, therefore, normal. 
\begin{figure}[ht]
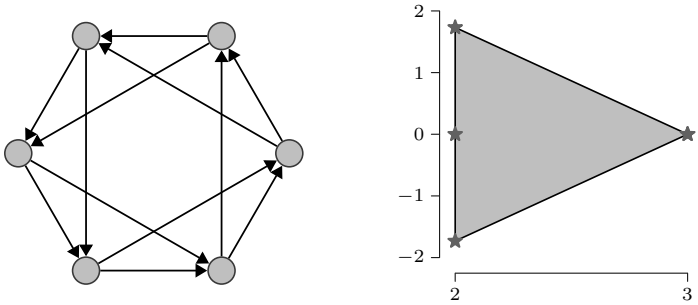

\centering
\drawFigure{circ-normal}{y min=-2,y max=2}
\caption{A normal digraph with circulant Laplacian matrix.}
\label{fig:circ-normal}
\end{figure}

Let $\Gamma=(V,E)\in\mathbb{G}$ be a normal digraph with Laplacian matrix $L$.
By~\cite[Condition 63]{Grone1987}, we have
\[
\norm{\mathbf{e}^{T}L}  = \norm{L\mathbf{e}} = 0,
\]
which implies that the column sums of $L$ are all zero.
Hence $\Gamma$ must be \emph{balanced}, that is, $d^{+}(i)=d^{-}(i)$ for all $i\in V$.
Figure~\ref{fig:bal-non-normal} shows a digraph that is balanced but does not have a normal Laplacian, which is evident from the shape of its restricted numerical range. 
\begin{figure}[ht]
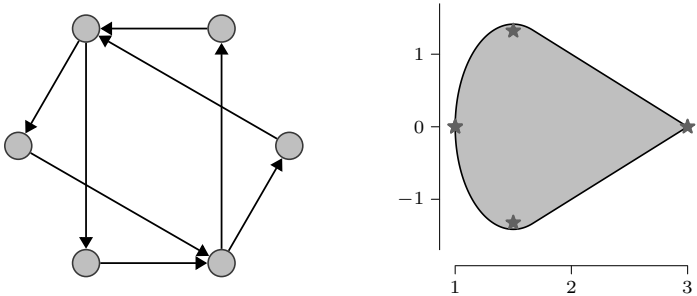

\centering
\drawFigure{bal-non-normal}{y min=-1.7,y max=1.7}
\caption{A balanced non-normal digraph.}
\label{fig:bal-non-normal}
\end{figure}
Therefore, the balanced property is, not surprisingly, weaker than the normal property for digraphs.
Nevertheless, the restricted numerical range of balanced digraphs is extremely well behaved under certain operations as described in Proposition~\ref{prop:bal-nr-ops}.
Note that the \emph{bidirectional join} of digraphs $\Gamma=\left(V,E\right)$ and $\Gamma' = \left(V',E'\right)$, where $V\cap V'=\emptyset$, is defined by
\[
\Gamma\bdjoin\Gamma' = \left(V\sqcup V',E\sqcup E'\sqcup\left\{(i,j),~(j,i)\colon i\in V, j\in V'\right\}\right)
\]
and the \emph{disjoint union} is defined by
\[
\Gamma\sqcup\Gamma' = \left(V\sqcup V',E\sqcup E'\right).
\]
Also, we use $I_{k}$ to denote the $k\times k$ identity matrix, and $J_{k_{1}\times k_{2}}$ to denote the $k_{1}\times k_{2}$ all ones matrix.
\begin{proposition}\label{prop:bal-nr-ops}
Let $\Gamma_{1},\Gamma_{2}\in\mathbb{G}$ be balanced digraphs of order $n_{1}$ and $n_{2}$, respectively.
Then, 
\begin{enumerate}[(i)]
\item $W_{r}\left(\Gamma_{1}\sqcup\Gamma_{2}\right) = \conv{W_{r}(\Gamma_{1})\cup W_{r}(\Gamma_{2})\cup\{0\}}$ 
\item $W_{r}\left(\Gamma_{1}\djoin\Gamma_{2}\right) =  \conv{\left(W_{r}(\Gamma_{1})+n_{2}\right)\cup W_{r}(\Gamma_{2})\cup\{n_{2}\}}$
\item $W_{r}\left(\Gamma_{1}\bdjoin\Gamma_{2}\right) = \conv{\left(W_{r}(\Gamma_{1})+n_{2}\right)\cup\left(W_{r}(\Gamma_{2})+n_{1}\right)\cup\{n_{1}+n_{2}\}}.$
\end{enumerate}
\end{proposition}
\begin{proof}
Let $L_{1}$ and $L_{2}$ denote the Laplacian matrix of $\Gamma_{1}$ and $\Gamma_{2}$, respectively. 
Also, let $Q_{1}$ and $Q_{2}$ be restrictor matrices of order $n_{1}$ and $n_{2}$, respectively.
Finally, define
\[
Q = \begin{bmatrix} Q_{1} & 0 & -\frac{n_{2}}{\sqrt{n_{2}^{2}n_{1}+n_{1}^{2}n_{2}}}\mathbf{e}^{n_{1}} \\ 0 & Q_{2} & \frac{n_{1}}{\sqrt{n_{2}^{2}n_{1}+n_{1}^{2}n_{2}}}\mathbf{e}^{n_{2}} \end{bmatrix}
\]
as a restrictor matrix of order $n=n_{1}+n_{2}$.
\begin{enumerate}[(i)]
\item The Laplacian matrix of $\Gamma_{1}\sqcup\Gamma_{2}$, possibly after re-ordering the vertices, can be written as 
\[
L = \begin{bmatrix} L_{1} & 0 \\ 0 & L_{2}\end{bmatrix}.
\]
Since both $\Gamma_{1}$ and $\Gamma_{2}$ are balanced, one can readily verify that
\[
Q^{*}LQ = \begin{bmatrix} Q_{1}^{*}L_{1}Q_{1} & 0 & 0 \\ 0 & Q_{2}^{*}L_{2}Q_{2} & 0 \\ 0 & 0 & 0\end{bmatrix}.
\]
Now, by~\cite[Property 1.2.10]{Horn1991} and Proposition~\ref{prop:basic-rnr}(i), it follows that
\[
W_{r}(L) = \conv{W_{r}(L_{1})\cup W_{r}(L_{2})\cup\{0\}}.
\]
\item The Laplacian matrix of $\Gamma_{1}\djoin\Gamma_{2}$, possibly after re-ordering the vertices, can be written as 
\[
L = \begin{bmatrix} L_{1}+n_{2}I_{n_{1}}& -J_{n_{1}\times n_{2}} \\ 0 & L_{2}\end{bmatrix}.
\]
Since both $\Gamma_{1}$ and $\Gamma_{2}$ are balanced, one can readily verify that
\[
Q^{*}LQ = \begin{bmatrix} Q_{1}^{*}L_{1}Q_{1}+n_{2}I_{n_{1}-1} & 0 & 0 \\ 0 & Q_{2}^{*}L_{2}Q_{2} & 0 \\ 0 & 0 & n_{2}\end{bmatrix}.
\]
By~\cite[Property 1.2.3 and 1.2.10]{Horn1991} and Proposition~\ref{prop:basic-rnr}, it follows that
\[
W_{r}(L) = \conv{\left(W_{r}(L_{1})+n_{2}\right)\cup W_{r}(L_{2})\cup\{n_{2}\}}.
\]
\item The Laplacian matrix of $\Gamma_{1}\bdjoin\Gamma_{2}$, possibly after re-ordering the vertices, can be written as 
\[
L = \begin{bmatrix} L_{1}+n_{2}I_{n_{1}} & -J_{n_{1}\times n_{2}} \\ -J_{n_{2}\times n_{1}} & L_{2}+n_{1}I_{n_{2}}\end{bmatrix}.
\]
Since both $\Gamma_{1}$ and $\Gamma_{2}$ are balanced, one can readily verify that
\[
Q^{*}LQ = \begin{bmatrix} Q_{1}^{*}L_{1}Q_{1} + n_{2}I_{n_{1}-1} & 0 & 0 \\ 0 & Q_{2}^{*}L_{2}Q_{2} + n_{1}I_{n_{2}-1} & 0 \\ 0 & 0 & n_{1}+n_{2}\end{bmatrix}.
\]
By~\cite[Property 1.2.3 and 1.2.10]{Horn1991} and Proposition~\ref{prop:basic-rnr}, it follows that
\[
W_{r}(L) = \conv{\left(W_{r}(L_{1})+n_{2}\right)\cup\left(W_{r}(L_{2})+n_{1}\right)\cup\{n_{1}+n_{2}\}}.
\]
\end{enumerate}
\end{proof}

Clearly, Proposition~\ref{prop:bal-nr-ops} holds for normal digraphs $\Gamma_{1}$ and $\Gamma_{2}$ since they are also balanced.
In this case, the proof of Proposition~\ref{prop:bal-nr-ops} implies that $Q^{*}LQ$ is normal for each Laplacian matrix $L$ of $\Gamma_{1}\circ\Gamma_{2}$, where $\circ\in\{\sqcup,\djoin,\bdjoin\}$.
Since the disjoint union and bidirectional join of two balanced digraphs are also balanced, Theorem~\ref{thm:rnrml&bal} implies that $\Gamma_{1}\sqcup\Gamma_{2}$ and $\Gamma_{1}\bdjoin\Gamma_{2}$ are both normal.
Therefore, Proposition~\ref{prop:bal-nr-ops} provides an easy way to construct normal digraphs from other normal digraphs in a way that the changes to the restricted numerical range are extremely well behaved. 
\begin{lemma}\label{lem:rnrml&bal}
Let $\Gamma\in\mathbb{G}$ have order $n$ and Laplacian matrix $L$.
Also, let $Q$ be a restrictor matrix of order $n$.
Then, $\Gamma$ is balanced if and only if $\mathbf{e}^{T}LQ=0$.
\end{lemma}
\begin{proof}
If $\Gamma$ is balanced, then $\mathbf{e}^{T}L=0$ and the result is trivial. 
Conversely, suppose that $\mathbf{e}^{T}LQ=0$. 
Let $\mathbf{q}_{1},\ldots,\mathbf{q}_{n-1}$ denote the columns of $Q$.
Since the entries of $\mathbf{e}^{T}L$ must sum to zero, it follows that $\mathbf{e}^{T}L$ is orthogonal to $\mathbf{e}$. 
Hence, there exist scalars $c_{1},\ldots,c_{n-1}$ such that
\[
\mathbf{e}^{T}L = c_{1}\mathbf{q}_{1}^{T} + \cdots + c_{n-1}\mathbf{q}_{n-1}^{T}
\]
and it follows that
\[
\mathbf{e}^{T}LQ = \left[c_{1}^{2},\ldots,c_{n-1}^{2}\right]^{T}. 
\]
Therefore, $\mathbf{e}^{T}LQ=0$ forces $c_{1}=\cdots=c_{n-1}=0$, that is, $\mathbf{e}^{T}L=0$, which implies that $\Gamma$ is balanced.
\end{proof}
\begin{theorem}\label{thm:rnrml&bal}
Let $\Gamma\in\mathbb{G}$ have order $n$ and Laplacian matrix $L$.
Also, let $Q$ be a restrictor matrix of order $n$.
Then, $L$ is normal if and only if $Q^{*}LQ$ is normal and $\Gamma$ is balanced. 
\end{theorem}
\begin{proof}
Define $\hat{Q}$ to be the $n\times n$ orthonormal matrix whose first $(n-1)$ columns are equal to the columns of $Q$ and whose last column is equal to $\mathbf{e}/\sqrt{n}$.
Then
\[
\hat{Q}^{*}L\hat{Q} = \begin{bmatrix} Q^{*}LQ & 0 \\ \frac{1}{\sqrt{n}}\mathbf{e}^{T}LQ & 0\end{bmatrix},
\]
which, by~\cite[Lemma 2.5.2]{Horn2013}, is normal if and only if $Q^{*}LQ$ is normal and $\mathbf{e}^{T}LQ = 0$.
By Lemma~\ref{lem:rnrml&bal}, $\mathbf{e}^{T}LQ=0$ if and only if $\Gamma$ is balanced. 
Therefore, $\hat{Q}^{*}L\hat{Q}$ is normal if and only if $Q^{*}LQ$ is normal and $\Gamma$ is balanced. 
\end{proof}

We conclude this section by providing another method for constructing normal digraphs that cannot be described as a disjoint union or bidirectional join of two (non-null) normal digraphs.
To this end, consider the following result.
\begin{proposition}\label{prop:nrml-identity}
Let $\Gamma=(V,E)\in\mathbb{G}$ have order $n$ and adjacency matrix $A=[a_{ij}]_{i,j=1}^{n}$.
Then $\Gamma$ is normal if and only if $\Gamma$ is balanced and
\begin{equation}\label{eq:nrml-identity}
\left(a_{ij} - a_{ji}\right)\left(d^{+}(j)-d^{+}(i)\right) = \sum_{k\neq i,j}\left(a_{ik}a_{jk} - a_{ki}a_{kj}\right),
\end{equation}
for all $i\neq j\in V$.
\end{proposition}
\begin{proof}
Let $L$ be the Laplacian matrix of $\Gamma$.
By~\cite[Condition 62]{Grone1987}, it follows that $L$ is normal if and only if
\[
\langle L\mathbf{e}_{i},L\mathbf{e}_{j}\rangle = \langle L^{T}\mathbf{e}_{i},L^{T}\mathbf{e}_{j}\rangle,
\]
for all $i,j\in V$, where $\mathbf{e}_{i}$ denotes the $i$th column of $I_{n}$.
Note that
\[
\langle L\mathbf{e}_{i},L\mathbf{e}_{j}\rangle
=
\begin{cases}
d^{+}(i)^{2} + d^{-}(i) & \text{if $i=j$,} \\
-d^{+}(i)a_{ij} - d^{+}(j)a_{ji} + \sum_{k\neq i,j}a_{ki}a_{kj} & \text{if $i\neq j$,}
\end{cases}
\]
and
\[
\langle L^{T}\mathbf{e}_{i},L^{T}\mathbf{e}_{j}\rangle
=
\begin{cases}
d^{+}(i)^{2} + d^{+}(i) & \text{if $i=j$,} \\
-d^{+}(i)a_{ji} - d^{+}(j)a_{ij} + \sum_{k\neq i,j}a_{ik}a_{jk} & \text{if $i\neq j$.}
\end{cases}
\]
Therefore, $L$ is normal if and only if $d^{+}(i)=d^{-}(i)$ for all $i\in V$ and
\[
\left(a_{ij}-a_{ji}\right)\left(d^{+}(j) - d^{+}(i)\right) = \sum_{k\neq i,j}\left(a_{ik}a_{jk}-a_{ki}a_{kj}\right),
\]
for all $i\neq j\in V$. 
\end{proof}

Given a graph $\Gamma = \left(V,E\right)$, we define the \emph{twin splitting} of a vertex $v\in V$ as the process in which a copy of $v$ is created, denoted $v'$, such that $v$ and $v'$ are \emph{twins}, that is, $(u,v)\in E$ if and only if $(u,v')\in E$ and $(v,u)\in E$ if and only if $(v',u)\in E$.
Note that the digraph in Figure~\ref{fig:bal-non-normal} can be constructed by performing a twin splitting on two non-adjacent vertices from a dicycle of order $4$.
It is worth emphasizing that altering the digraph in this way preserves the balanced property but destroys normality.
Indeed, let $i$ and $j$ be vertices in Figure~\ref{fig:bal-non-normal} such that $i$ points at $j$, $i$ has a twin, and $j$ does not have a twin. 
It follows that the left side of~\eqref{eq:nrml-identity} is equal to $1$ and the right side of~\eqref{eq:nrml-identity} is equal to $0$, which confirms that the digraph in Figure~\ref{fig:bal-non-normal} is not normal.

It turns out that we can restore the normality of the digraph in Figure~\ref{fig:bal-non-normal} by adding a new vertex that is bidirectionally joined onto all the twin vertices, as in Figure~\ref{fig:nrml-twin-splitting}.
This process is generalized in Theorem~\ref{thm:nrml-twin-splitting}.
\begin{figure}[ht]
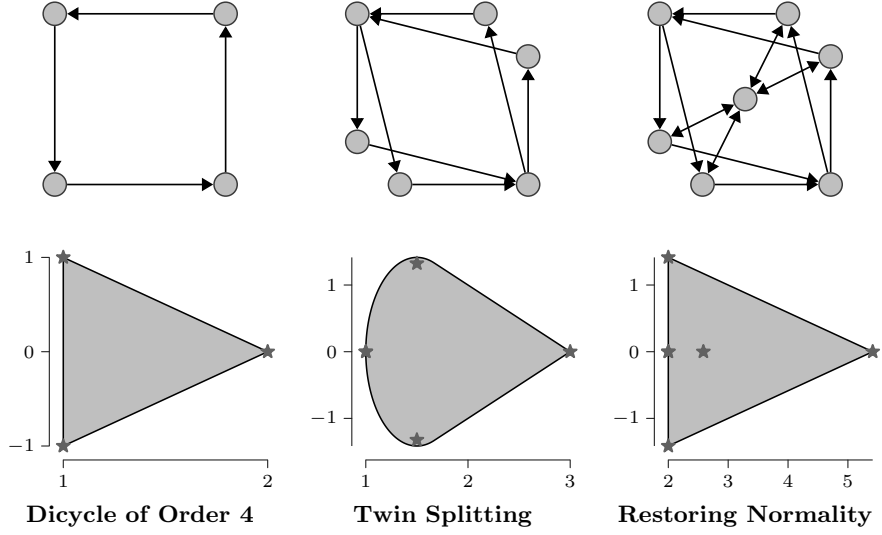

\centering
\tikzset{/my/graphic options/tall preset}
\begin{tabular}{ccc}
\drawFigure{4-cycle}{}
&
\drawFigure{twin-splitting}{}
&
\drawFigure{nrml-rest}{}
\\
\textbf{\small{Dicycle of Order 4}}
&
\textbf{\small{Twin Splitting}}
&
\textbf{\small{Restoring Normality}}
\end{tabular}
\caption{Twin splitting and restoration of normality.}
\label{fig:nrml-twin-splitting}
\end{figure}

\begin{theorem}\label{thm:nrml-twin-splitting}
Let $\Gamma=(V,E)\in\mathbb{G}$ be a dicycle of order $n\geq 4$, where $n$ is even.
Starting at any vertex in $V$, construct $\Gamma'$ by performing a twin splitting on every other vertex in the dicycle. 
Also, let $\Gamma''$ be constructed from $\Gamma'$ by adding a new vertex that is bidirectionally joined to all the twin vertices in $\Gamma'$. 
Then, $\Gamma'$ is a balanced digraph that is not normal and $\Gamma''$ is a normal digraph. 
\end{theorem}
\begin{proof}
Let $\Gamma' = (V',E')$ and note that $\abs{V'} = 3n/2$, where $n$ vertices have twins and the other $n/2$ vertices don't have twins. 
Consider the partition of the set of vertices
\[
V' = T\sqcup V'\setminus{T},
\]
where $T$ is the set of vertices with twins and $V'\setminus{T}$ is the set of vertices without twins. 
Note that
\[
d_{\Gamma'}^{+}(i) = d_{\Gamma'}^{-}(i) = 1,
\]
for all $i\in T$ and
\[
d_{\Gamma'}^{+}(i) = d_{\Gamma'}^{-}(i) = 2,
\]
for all $i\in V'\setminus{T}$.
Therefore, the digraph $\Gamma'$ is clearly balanced.
Furthermore, for each vertex $i\in T$, there exists a vertex $j\in V'\setminus{T}$ such that $a_{ij}=1$ and $a_{ji}=0$, that is, vertex $i$ points at vertex $j$ and vertex $j$ does not point at vertex $i$.
Therefore, 
\[
\left(a_{ij}-a_{ji}\right)\left(d^{+}(j)-d^{+}(i)\right) = 1.
\]
Since $i$ and $j$ don't point at any of the same vertices, and no vertex points at both $i$ and $j$, we have
\[
\sum_{k\neq i,j}\left(a_{ik}a_{jk} - a_{ki}a_{kj}\right) = 0,
\]
and it follows that~\eqref{eq:nrml-identity} does not hold; hence, $\Gamma'$ is not normal. 

Now, let $\Gamma'' = (V'',E'')$ and consider the partition of the set of vertices
\[
V'' = T\sqcup V'\setminus{T}\sqcup\{v\},
\]
where $v$ is bidirectionally joined onto all vertices in $T$. 
Therefore,
\[
d_{\Gamma''}^{+}(i) = d_{\Gamma''}^{-}(i) = 2,
\]
for all $i\in T\sqcup V'\setminus{T}$.
Moreover, 
\[
d_{\Gamma''}^{+}(v) = d_{\Gamma''}^{-}(v) = \abs{T},
\]
and it follows that $\Gamma''$ is balanced.
It is also clear that~\eqref{eq:nrml-identity} holds for all $i\in T$ and $j\in V'\setminus{T}$ since those vertices have the same out-degree and $a_{ik}a_{jk}=a_{ki}a_{kj}=0$ for all $k\neq i,j$.

If $i,j\in T$, then the vertices have the same out-degree and there exists $k,k'\neq i,j$ such that $a_{ik}a_{jk}=1$ and $a_{k'i}a_{k'j}=1$. 
Also, the vertex $v$ satisfies $a_{iv}a_{jv}=a_{vi}a_{vj}=1$, and for all other $k\neq i,j$, we have $a_{ik}=a_{jk}=a_{ki}a_{kj}=0$.
Hence, the terms on the right side of~\eqref{eq:nrml-identity} sum to zero.
A similar result holds for vertices $i,j\in V'\setminus{T}$ and for $j=v,i\neq v$.
Therefore, both sides of~\eqref{eq:nrml-identity} are equal to zero for all $i,j\in V$, and it follows that $\Gamma''$ is normal.
\end{proof}

\subsection{Restricted-Normal Digraphs}\label{subsec:analysis-rnrml}
Let $\Gamma_{1}$ and $\Gamma_{2}$ be normal digraphs of order $n_{1}$ and $n_{2}$, respectively. 
Then, from the proof of Proposition~\ref{prop:bal-nr-ops}, it follows that $Q^{*}LQ$ is normal, where $L$ is the Laplacian matrix of $\Gamma_{1}\djoin\Gamma_{2}$ and $Q$ is a restrictor matrix of order $n=n_{1}+n_{2}$.
Moreover, by Theorem~\ref{thm:rnrml&bal}, it follows that $L$ is not normal since $\Gamma_{1}\djoin\Gamma_{2}$ is not balanced. 
Therefore, the directed join of two normal digraphs results in a restricted-normal digraph as illustrated in Figure~\ref{fig:classes}. 
Moreover, by Proposition~\ref{prop:rnrml-djoin}, the directed join of two normal digraphs is the only restricted-normal digraph that can be described by a directed join of two digraphs. 
\begin{proposition}\label{prop:rnrml-djoin}
Let $\Gamma_{1},\Gamma_{2}\in\mathbb{G}$ have order $n_{1}$ and $n_{2}$, respectively.
Then, $\Gamma_{1}\djoin\Gamma_{2}$ is a restricted-normal digraph if and only if $\Gamma_{1}$ and $\Gamma_{2}$ are normal digraphs.
\end{proposition}
\begin{proof}
Let $L_{1}$ and $L_{2}$ denote the Laplacian matrix of $\Gamma_{1}$ and $\Gamma_{2}$, respectively. 
Also, let $Q_{1}$ and $Q_{2}$ be restrictor matrices of order $n_{1}$ and $n_{2}$, respectively. 
Finally, define
\[
Q = \begin{bmatrix} Q_{1} & 0 & -\frac{n_{2}}{\sqrt{n_{2}^{2}n_{1}+n_{1}^{2}n_{2}}}\mathbf{e}^{n_{1}} \\ 0 & Q_{2} & \frac{n_{1}}{\sqrt{n_{2}^{2}n_{1}+n_{1}^{2}n_{2}}}\mathbf{e}^{n_{2}} \end{bmatrix}
\]
as a restrictor matrix of order $n=n_{1}+n_{2}$.
The Laplacian matrix of $\Gamma_{1}\djoin\Gamma_{2}$, possibly after re-ordering the vertices, can be written as
\[
L = \begin{bmatrix} L_{1} + n_{2}I_{n_{1}} & -J_{n_{1}\times n_{2}} \\ 0 & L_{2} \end{bmatrix}.
\]
Now, one can readily verify that
\[
Q^{*}LQ = \begin{bmatrix}Q_{1}^{*}L_{1}Q_{1} + n_{2}I_{n_{1}-1} & 0 & 0 \\ 0 & Q_{2}^{*}L_{2}Q_{2} & 0 \\ \frac{-n_{2}}{\sqrt{n_{2}^{2}n_{1}+n_{1}^{2}n_{2}}}\left(\mathbf{e}^{n_{1}}\right)^{T}L_{1}Q_{1} & \frac{n_{1}}{\sqrt{n_{2}^{2}n_{1}+n_{1}^{2}n_{2}}}\left(\mathbf{e}^{n_{2}}\right)^{T}L_{2}Q_{2} & n_{2} \end{bmatrix}.
\]
By~\cite[Lemma 2.5.2]{Horn2013}, $Q^{*}LQ$ is normal if and only if $Q_{1}^{*}L_{1}Q_{1}$ and $Q_{2}^{*}L_{2}Q_{2}$ are normal and $\left(\mathbf{e}^{n_{1}}\right)^{T}L_{1}Q_{1}=0$ and $\left(\mathbf{e}^{n_{2}}\right)^{T}L_{2}Q_{2}=0$.
Furthermore, by Lemma~\ref{lem:rnrml&bal}, $\left(\mathbf{e}^{n_{1}}\right)^{T}L_{1}Q_{1}=0$ and $\left(\mathbf{e}^{n_{2}}\right)^{T}L_{2}Q_{2}=0$ if and only if $\Gamma_{1}$ and $\Gamma_{2}$ are balanced.
Thus, by Theorem~\ref{thm:rnrml&bal}, $Q^{*}LQ$ is normal if and only if $L_{1}$ and $L_{2}$ are normal. 
\end{proof}

Let $\Gamma = (V,E)\in\mathbb{G}$ have order $n$.
If $n$ is square-free or $n=2m$, where $m$ is square-free, it turns out that $\Gamma$ is restricted-normal if and only if it is a directed join. 
To help understand this result, we define the \emph{imbalance} of vertex $i\in V$ by
\[
\iota(i) = d^{+}(i) - d^{-}(i).
\]
\begin{lemma}\label{lem:imb-djoin}
Let $\Gamma=(V,E)\in\mathbb{G}$ be a digraph of order $n$ that is not balanced.
Then, $\Gamma$ is a directed join of two balanced digraphs if and only if
\begin{equation}\label{eq:imb-djoin}
n\mid\left(\iota(i)-\iota(j)\right),
\end{equation}
for all $i,j\in V$.
\end{lemma}
\begin{proof}
Suppose that $\Gamma=\Gamma_{1}\djoin\Gamma_{2}$, where $\Gamma_{1} = (V_{1},E_{1})$ and $\Gamma_{2} = (V_{2},E_{2})$ are balanced digraphs.
Then, the imbalance of vertex $i\in V=V_{1}\sqcup V_{2}$ satisfies
\[
\iota(i) = \begin{cases} \abs{V_{2}} & \text{if $i\in V_{1}$} \\ -\abs{V_{1}} & \text{if $i\in V_{2}$} \end{cases}.
\]
Therefore,
\[
\iota(i)-\iota(j) = \begin{cases}-n & \text{if $i\in V_{2}$ and $j\in V_{1}$} \\ 0 & \text{if $i,j\in V_{1}$ or $i,j\in V_{2}$} \\ n & \text{if $i\in V_{1}$ and $j\in V_{2}$} \end{cases},
\]
and it follows that $n\mid\left(\iota(i)-\iota(j)\right)$, for all $i,j\in V$.

Conversely, suppose that $n\mid(\left(\iota(i)-\iota(j)\right)$, for all $i,j\in V$.
Since $\abs{\iota(i)}\leq(n-1)$ for all $i\in V$, it follows that
\begin{equation}\label{eq:imb-djoin1}
\iota(i)-\iota(j)\in\{-n,0,n\},
\end{equation}
for all $i,j\in V$. 
Since $\Gamma$ is not balanced and the sum of imbalances over all vertices is zero, there exists a vertex $u\in V$ such that $\iota(u)>0$ and there exists a vertex $v\in V$ such that $\iota(v)<0$. 
Thus,~\eqref{eq:imb-djoin1} implies that $\iota(u)-\iota(v)=n$, and we can partition the vertex set as $V=V_{1}\sqcup V_{2}$, where $\iota(i)=\iota(u)$ for all $i\in V_{1}$ and $\iota(i)=\iota(v)$ for all $i\in V_{2}$.
Summing imbalances, we have
\[
\iota(u)\abs{V_{1}} + \iota(v)\abs{V_{2}} = \iota(u)\abs{V_{1}} + (\iota(u)-n)\abs{V_{2}} = 0,
\]
which implies that $\iota(u)=\abs{V_{2}}$ and $(\iota(u)-n)=-\abs{V_{1}}$.

Note that the sum of the imbalances in $V_{2}$ is equal to the total out-degree minus the total in-degree of the vertices in $V_{2}$, which is also equal to the number of edges pointing from $V_{2}$ to $V_{1}$ minus the number of edges pointing from $V_{1}$ to $V_{2}$.
Since the sum of the imbalances in $V_{2}$ is equal to $-\abs{V_{1}}\abs{V_{2}}$, it follows that there are no edges pointing from $V_{2}$ to $V_{1}$ and there is an edge pointing from every vertex in $V_{1}$ to all vertices in $V_{2}$.
Hence, $\Gamma$ is a directed join of two digraphs: $\Gamma_{1} = (V_{1},E_{1})$ and $\Gamma_{2} = (V_{2},E_{2})$. 
Finally, note that the imbalance of each vertex in $V$ would be zero if all edges pointing from $V_{1}$ to $V_{2}$ were removed; thus, $\Gamma_{1}$ and $\Gamma_{2}$ are balanced.
\end{proof}

Lemma~\ref{lem:imb-djoin} shows that directed joins of balanced digraphs can be identified by their vertex imbalances. 
The connection between this result and restricted-normal digraphs will be made using the following propositions.
First, note that given a restrictor matrix $Q$ of order $n$, we denote
\begin{equation}\label{eq:projector}
P = QQ^{*} = I - \frac{1}{n}\mathbf{e}\mathbf{e}^{T}
\end{equation}
as the unique projector whose image space is $\mathbf{e}^{\perp}$ and kernel space is $\spn{\mathbf{e}}$. 
\begin{proposition}\label{prop:nrml-proj}
Let $\Gamma=(V,E)\in\mathbb{G}$ have order $n$ and Laplacian matrix $L$.
Then, $Q^{*}LQ$ is normal if and only if $PL$ is normal.
\end{proposition}
\begin{proof}
Let $Q$ be a restrictor matrix of order $n$.
Then, by~\cite[Condition 62]{Grone1987}, $Q^{*}LQ$ is normal if and only if
\begin{equation}\label{eq:nrml-proj-p1}
\langle Q^{*}LQ\mathbf{x},Q^{*}LQ\mathbf{y}\rangle = \langle Q^{*}L^{T}Q\mathbf{x},Q^{*}L^{T}Q\mathbf{y}\rangle,
\end{equation}
for all $\mathbf{x},\mathbf{y}\in\mathbb{C}^{n-1}$.
Since $Q$ is a bijection between $\mathbb{C}^{n-1}$ and $\mathbf{e}^{\perp}$, it follows that~\eqref{eq:nrml-proj-p1} holds for all $\mathbf{x},\mathbf{y}\in\mathbb{C}^{n-1}$ if and only if 
\begin{equation}\label{eq:nrml-proj-p2}
\langle Q^{*}L\mathbf{x},Q^{*}L\mathbf{y}\rangle = \langle Q^{*}L^{T}\mathbf{x},Q^{*}L^{T}\mathbf{y}\rangle,
\end{equation}
for all $\mathbf{x},\mathbf{y}\in\mathbf{e}^{\perp}$.
Using the projector $P$ from~\eqref{eq:projector}, we can re-write~\eqref{eq:nrml-proj-p2} as
\[
\langle PL\mathbf{x},L\mathbf{y}\rangle = \langle PL^{T}\mathbf{x},L^{T}\mathbf{y}\rangle. 
\]
Furthermore, since the image space of $L^{T}$ is equal to $\mathbf{e}^{\perp}$, it follows that $Q^{*}LQ$ is normal if and only if
\begin{equation}\label{eq:nrml-proj-p3}
\langle PL\mathbf{x},L\mathbf{y}\rangle = \langle L^{T}\mathbf{x},L^{T}\mathbf{y}\rangle,
\end{equation}
for all $\mathbf{x},\mathbf{y}\in\mathbf{e}^{\perp}$.
Note that $\langle PL\mathbf{x},L\mathbf{y}\rangle = \langle PL\mathbf{x},PL\mathbf{y}\rangle$; hence,~\eqref{eq:nrml-proj-p3} holds for all $\mathbf{x},\mathbf{y}\in\mathbf{e}^{\perp}$ if and only if
\[
\langle PL\mathbf{x},PL\mathbf{y}\rangle = \langle L^{T}P\mathbf{x},L^{T}P\mathbf{y}\rangle 
\]
for all $\mathbf{x},\mathbf{y}\in\mathbb{C}^{n}$, that is, if and only if $PL$ is normal. 
\end{proof}
\begin{proposition}\label{prop:rnrml-identity}
Let $\Gamma=(V,E)\in\mathbb{G}$ have order $n$ and Laplacian matrix $L$.
Then, for any $k\in V$, $\Gamma$ is restricted-normal if and only if $\Gamma$ is not balanced and
\begin{equation}\label{eq:rnrm-identity}
\frac{\left(\iota(i)-\iota(k)\right)\left(\iota(j)-\iota(k)\right)}{n} = \langle L(\mathbf{e}_{i}-\mathbf{e}_{k}),L(\mathbf{e}_{j}-\mathbf{e}_{k})\rangle -  \langle L^{T}(\mathbf{e}_{i}-\mathbf{e}_{k}),L^{T}(\mathbf{e}_{j}-\mathbf{e}_{k})\rangle,
\end{equation}
for all $i,j\in V\setminus{\{k\}}$.
\end{proposition}
\begin{proof}
Let $\left\{\mathbf{x}_{i}\right\}_{i\in\alpha}$ be a basis for $\mathbf{e}^{\perp}$, where $\alpha$ is an arbitrary index set of cardinality $(n-1)$. 
Also, let $P$ be the projector from~\eqref{eq:projector}.
Then, one can readily verify that~\eqref{eq:nrml-proj-p3} implies that $Q^{*}LQ$ is normal if and only if
\[
\langle PL\mathbf{x}_{i},L\mathbf{x}_{j}\rangle = \langle L^{T}\mathbf{x}_{i},L^{T}\mathbf{x}_{j}\rangle,
\]
for all $i,j\in\alpha$.
Furthermore, the set $\left\{\mathbf{e}_{i}-\mathbf{e}_{k}\right\}_{i\in V\setminus{\{k\}}}$ clearly forms a basis for $\mathbf{e}^{\perp}$, which implies that $Q^{*}LQ$ is normal if and only if
\[
\langle PL(\mathbf{e}_{i}-\mathbf{e}_{k}),L(\mathbf{e}_{j}-\mathbf{e}_{k})\rangle = \langle L^{T}(\mathbf{e}_{i}-\mathbf{e}_{k}),L^{T}(\mathbf{e}_{j}-\mathbf{e}_{k})\rangle,
\]
for all $i,j\in V\setminus{\{k\}}$.
Since $P=I - \frac{1}{n}\mathbf{e}\mathbf{e}^{T}$, it follows that $Q^{*}LQ$ is normal if and only if~\eqref{eq:rnrm-identity} holds, for all $i,j\in V\setminus{\{k\}}$.
By Theorem~\ref{thm:rnrml&bal}, $\Gamma$ is restricted-normal if and only if $\Gamma$ is not balanced and~\eqref{eq:rnrm-identity} holds, for all $i,j\in V\setminus{\{k\}}$.
\end{proof}

We are now ready to prove that if $n$ is square-free or $n=2m$, where $m$ is square-free, then the only restricted-normal digraphs are directed joins.
\begin{theorem}\label{thm:sqfree-rnrml}
Let $\Gamma\in\mathbb{G}$ have order $n$, where $n$ is square-free or $n=2m$ with $m$ square-free. 
If $\Gamma$ is a restricted-normal digraph, then $\Gamma$ is a directed join of normal digraphs. 
\end{theorem}
\begin{proof}
Suppose that $\Gamma = (V,E)$ is restricted-normal and has Laplacian  matrix $L$.
Then, by Proposition~\ref{prop:rnrml-identity}, $\Gamma$ is not balanced and
\begin{equation}\label{eq:sqfree-rnrml1}
\frac{\left(\iota(i)-\iota(j)\right)^{2}}{n} = \norm{L(\mathbf{e}_{i}-\mathbf{e}_{j})}^{2} - \norm{L^{T}(\mathbf{e}_{i}-\mathbf{e}_{j})}^{2},
\end{equation}
for all $i,j\in V$.
Since the right side of~\eqref{eq:sqfree-rnrml1} is integer valued, we have $n\mid\left(\iota(i)-\iota(j)\right)^{2}$, for all $i,j\in V$.
Furthermore, if $n$ is square-free, then we have $n\mid\left(\iota(i)-\iota(j)\right)$, for all $i,j\in V$, and it follows from Lemma~\ref{lem:imb-djoin} that $\Gamma$ is a directed join. 
Moreover, since $\Gamma$ is restricted-normal, Proposition~\ref{prop:rnrml-djoin} implies that $\Gamma$ is a directed join of normal digraphs. 

Let $L=[l_{ij}]_{i,j=1}^{n}$ denote the Laplacian matrix of $\Gamma$.
Then, the right side of~\eqref{eq:sqfree-rnrml1} can be written as
\begin{align*}
 \norm{L(\mathbf{e}_{i}-\mathbf{e}_{j})}^{2} - \norm{L^{T}(\mathbf{e}_{i}-\mathbf{e}_{j})}^{2} &= \sum_{k=1}^{n}\left(l_{ki}-l_{kj}\right)^{2} - \left(l_{ik}-l_{jk}\right)^{2} \\
 &= \sum_{k=1}^{n}\left(l_{ki}^{2}-l_{ik}^{2}\right) + \left(l_{jk}^{2}-l_{kj}^{2}\right) - 2\left(l_{ki}l_{kj}-l_{ik}l_{jk}\right) \\
 &=-\iota(i) - \iota(j) - 2\sum_{k=1}^{n}\left(l_{ki}l_{kj}-l_{ik}l_{jk}\right).
\end{align*}
Therefore,~\eqref{eq:sqfree-rnrml1} can be re-written as
\begin{equation}\label{eq:sqfree-rnrml2}
\frac{\left(\iota(i)-\iota(j)\right)^{2}}{n} = -\left(\iota(i)+\iota(j)\right) - 2\sum_{k=1}^{n}\left(l_{ki}l_{kj}-l_{ik}l_{jk}\right).
\end{equation}

Now, suppose that $n=2m$, where $m$ is square-free.
Then,
\[
2\mid n\mid\left(\iota(i)-\iota(j)\right)^{2},
\]
which implies that $2\mid\left(\iota(i)-\iota(j)\right)$. 
Thus, we have $2\mid\left(\iota(i)+\iota(j)\right)$, and it follows that there exists an integer $s$ such that $2s=\left(\iota(i)+\iota(j)\right)$.
Therefore, we can re-write~\eqref{eq:sqfree-rnrml2} as
\[
\frac{\left(\iota(i)-\iota(j)\right)^{2}}{2m} = -2s - 2\sum_{k=1}^{n}\left(l_{ki}l_{kj}-l_{ik}l_{jk}\right),
\]
that is,
\[
\frac{\left(\iota(i)-\iota(j)\right)^{2}}{4m} = -s - \sum_{k=1}^{n}\left(l_{ki}l_{kj}-l_{ik}l_{jk}\right).
\]
Hence, $m\mid\left(\iota(i)-\iota(j)\right)^{2}/4$, which implies that $m\mid\left(\iota(i)-\iota(j)\right)/2$, that is, $n\mid\left(\iota(i)-\iota(j)\right)$.
Therefore, if $\Gamma$ is restricted-normal and $n=2m$, where $m$ is square-free, then $n\mid\left(\iota(i)-\iota(j)\right)$, for all $i,j\in V$, and Lemma~\ref{lem:imb-djoin} implies that $\Gamma$ is a directed join.
Moreover, since $\Gamma$ is restricted-normal, Proposition~\ref{prop:rnrml-djoin} implies that $\Gamma$ is a directed join of normal digraphs. 
\end{proof}

According to Theorem~\ref{thm:sqfree-rnrml}, $n=8$ is the smallest possible order for which there could exist restricted-normal digraphs that are not directed joins. 
In fact, $48$ of the $4349$ restricted-normal digraphs of order $8$ (see Table~\ref{tab:classes}) are not directed joins.
For example, consider the digraph in Figure~\ref{fig:rnrml-non-djoin}, which is clearly not a directed join based on the imbalances of its vertices, Proposition~\ref{prop:rnrml-djoin}, and Lemma~\ref{lem:imb-djoin}.
\begin{figure}[ht]
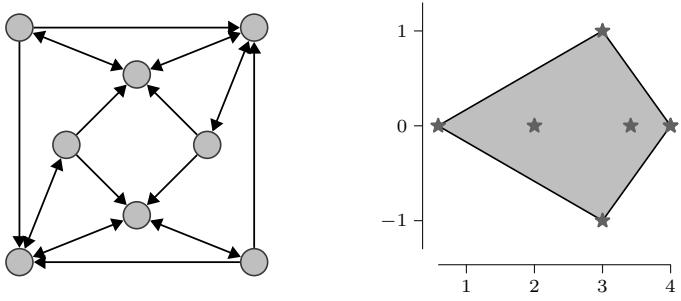

\centering
\drawFigure{rnrml-non-djoin}{y min=-1.3,y max=1.3}
\caption{A restricted-normal digraph of order $8$ that is not a directed join.}
\label{fig:rnrml-non-djoin}
\end{figure}

In what follows, we provide a construction for restricted-normal digraphs, where the order is neither square-free nor twice a square-free number, that are not directed joins. 
To this end, we make use of the following lemma which shows that restricted-normal digraphs have Laplacian matrices with a particular left-eigenvector structure. 

\begin{lemma}\label{lem:rnrm-leigvecs}
Let $\Gamma\in\mathbb{G}$ have order $n$ and Laplacian matrix $L$.
Then, $\Gamma$ is a restricted-normal digraph if and only if $\Gamma$ is not balanced and there exist orthogonal eigenvectors $\mathbf{u}_{1},\ldots,\mathbf{u}_{n-1}\in\mathbf{e}^{\perp}$ for $L^{T}$.
\end{lemma}
\begin{proof}
Let $P$ be the projector from~\eqref{eq:projector}.
By Theorem~\ref{thm:rnrml&bal} and Proposition~\ref{prop:nrml-proj}, $\Gamma$ is restricted-normal if and only if $\Gamma$ is not balanced and $PL$ is normal. 
Furthermore, by~\cite[Condition 12 and Condition 14]{Grone1987}, $PL$ is normal if and only if there exists orthogonal eigenvectors $\mathbf{u}_{1},\ldots,\mathbf{u}_{n-1},\mathbf{u}_{n}\in\mathbb{C}^{n}$ for $L^{T}P$.
Without loss of generality, we assume that $\mathbf{u}_{n}=\mathbf{e}$. 
Hence, $\mathbf{u}_{1},\ldots,\mathbf{u}_{n-1}\in\mathbf{e}^{\perp}$ are orthogonal eigenvectors of $L^{T}$, and the result follows.
\end{proof}

Next, we provide a construction for restricted-normal digraphs of order $n^{2}$, where $n\geq 3$, that are not directed joins.
Note that we use $J_{k}$ to denote the $k\times k$ all ones matrix. 
\begin{theorem}\label{thm:rnrm-ndjoin}
Let $\Gamma\in\mathbb{G}$ have order $n^{2}$, where $n\geq 3$, and adjacency matrix
\[
A = \begin{bmatrix} & I_{n} & \cdots & I_{n} \\ J_{n}-I_{n} & & & \\ \vdots & & & \\ J_{n}-I_{n} \end{bmatrix},
\]
where all blank entries are zero.
Then, $\Gamma$ is a restricted-normal digraph that is not a directed join.
\end{theorem}
\begin{proof}
Note that the Laplacian matrix of $\Gamma$ is equal to
\[
L = (n-1)I_{n^{2}} - A.
\]
Let $\mathbf{x}_{1},\ldots,\mathbf{x}_{n}\in\mathbb{C}^{n}$ satisfy
\[
\left(J_{n}-I_{n}\right)\mathbf{x}_{1} = \mathbf{e}^{n}
~\text{and}~
\mathbf{x}_{2}+\cdots+\mathbf{x}_{n}=\mathbf{e}^{n}.
\]
Then, the vector $\mathbf{x} = \left[\mathbf{x}_{1},\cdots,\mathbf{x}_{n}\right]^{T}$ satisfies $A\mathbf{x} = \mathbf{e}^{n^{2}}$. 
Hence, every vector in the null space of $A^{T}$ is orthogonal to $\mathbf{e}^{n^{2}}$. 
Furthermore, the nullity of $A^{T}$ is equal to $n(n-2)$, which implies the existence of $n(n-2)$ orthogonal eigenvectors for $L^{T}$, corresponding to the eigenvalue $(n-1)$, that are perpendicular to $\mathbf{e}^{n^{2}}$.

Now, let $\mathbf{x}_{1},\ldots,\mathbf{x}_{n}\in\mathbb{C}^{n}$ be selected so that $\mathbf{x} = \left[\mathbf{x}_{1},\cdots,\mathbf{x}_{n}\right]^{T}$ is an eigenvector for $A^{T}$ corresponding to a non-zero eigenvalue $\lambda$. 
Then,
\begin{equation}\label{eq:rnrm-ndjoin1}
A^{T}\mathbf{x} =
\begin{bmatrix} (J_{n}-I_{n})\mathbf{x}_{2} + \cdots + (J_{n}-I_{n})\mathbf{x}_{n} \\ \mathbf{x}_{1} \\ \vdots \\ \mathbf{x}_{1} \end{bmatrix}
=
\lambda
\begin{bmatrix} \mathbf{x}_{1} \\ \mathbf{x}_{2} \\ \vdots \\ \mathbf{x}_{n} \end{bmatrix},
\end{equation}
which implies that $\mathbf{x}_{2}=\cdots=\mathbf{x}_{n} = \lambda^{-1}\mathbf{x}_{1}$ and
\begin{equation}\label{eq:rnrm-ndjoin2}
(J_{n}-I_{n})\mathbf{x}_{1} = \frac{\lambda^{2}}{n-1}\mathbf{x}_{1}.
\end{equation}
Hence, $\mathbf{x}_{1}$ is an eigenvector of $(J_{n}-I_{n})$. 

Note that $(J_{n}-I_{n})$ is a symmetric matrix with eigenvalues
\[
\lambda_{1}=(n-1),\lambda_{2}=-1,\ldots,\lambda_{n}=-1
\]
and corresponding orthogonal eigenvectors that we denote by $\mathbf{v}_{1},\ldots,\mathbf{v}_{n}$.
Also,~\eqref{eq:rnrm-ndjoin1} and~\eqref{eq:rnrm-ndjoin2} imply that
\[
\mu_{2k-1}=+\sqrt{(n-1)\lambda_{k}},~\mu_{2k}=-\sqrt{(n-1)\lambda_{k}}
\]
are eigenvalues of $A^{T}$ with corresponding orthogonal eigenvectors
\[
\mathbf{u}_{2k-1} = \begin{bmatrix}\mu_{2k-1}\mathbf{v}_{k} \\ \mathbf{v}_{k} \\ \vdots \\ \mathbf{v}_{k} \end{bmatrix},
~\mathbf{u}_{2k} = \begin{bmatrix}\mu_{2k}\mathbf{v}_{k} \\ \mathbf{v}_{k} \\ \vdots \\ \mathbf{v}_{k} \end{bmatrix},
\]
for $k=1,2,\ldots,n$.

Since $\mathbf{v}_{1}$ is a constant multiple of $\mathbf{e}^{n}$, it follows that $\mathbf{u}_{2},\ldots,\mathbf{u}_{2n}$ are perpendicular to $\mathbf{e}^{n^{2}}$.
Furthermore, $\mathbf{u}_{2},\ldots,\mathbf{u}_{2n}$ are in the column space of $A$, so these vectors are orthogonal to every null vector of $A^{T}$. 
Therefore, there are $n^{2}-1$ orthogonal eigenvectors of $L^{T}$ that are perpendicular to $\mathbf{e}^{n^{2}}$, and Lemma~\ref{lem:rnrm-leigvecs} implies that $\Gamma$ is restricted-normal. 
Furthermore, since $\Gamma$ is clearly not balanced and its imbalances don't satisfy~\eqref{eq:imb-djoin}, Proposition~\ref{prop:rnrml-djoin} and Lemma~\ref{lem:imb-djoin} imply that $\Gamma$ is not a directed join.
\end{proof}

Next, we use the Kronecker product to construct new restricted-normal digraphs of order $nk$, where $k\geq 1$, from old restricted-normal digraphs of order $n$.
\begin{theorem}\label{thm:rnrm-tprod}
Let $\Gamma=(V,E)\in\mathbb{G}$ be a restricted-normal digraph of order $n$ with adjacency matrix $A$.
For $k\geq 1$, the digraph $\hat{\Gamma}=(\hat{V},\hat{E})\in\mathbb{G}$ associated with the adjacency matrix $A\otimes J_{k}$ is a restricted-normal digraph of order $nk$. 
Furthermore, $\hat{\Gamma}$ is a directed join if and only if $\Gamma$ is a directed join. 
\end{theorem}
\begin{proof}
Note that the Laplacian matrix of $\Gamma$ can be written as $L = D - A$, where $D=\diag\left(d^{+}(1),\ldots,d^{+}(n)\right)$.
Also, the vertex set of $\hat{\Gamma}$ can be written as $\hat{V} = \left\{i_{j}\colon~i\in V,~j=1,\ldots,k\right\}$ and the Laplacian matrix of $\hat{\Gamma}$ is equal to
\[
\hat{L} = k\left(D\otimes I_{k}\right) - \left(A\otimes J_{k}\right).
\]
Since the imbalance of vertex $i_{j}\in\hat{V}$ is equal to the corresponding column sum of $\hat{L}$, we have
\begin{equation}\label{eq:rnrm-tprod1}
\iota_{\hat{\Gamma}}(i_{j}) = kd_{\Gamma}^{+}(i) - kd_{\Gamma}^{-}(i) = k\cdot\iota_{\Gamma}(i),
\end{equation}
for all $i\in V$ and $j=1,\ldots,k$.

Now, since $\Gamma$ is a restricted-normal digraph, Lemma~\ref{lem:rnrm-leigvecs} implies that there exist orthogonal eigenvectors $\mathbf{u}_{1},\ldots,\mathbf{u}_{n-1}$ for $L^{T}$ that are perpendicular to $\mathbf{e}^{n}$. 
Therefore, for $i=1,\ldots,n-1$, we have
\begin{align*}
\hat{L}^{T}\left(\mathbf{u}_{i}\otimes\mathbf{e}^{k}\right) &= k\left(D^{T}\otimes I_{k}\right)\left(\mathbf{u}_{i}\otimes\mathbf{e}^{k}\right) - \left(A^{T}\otimes J_{k}\right)\left(\mathbf{u}_{i}\otimes\mathbf{e}^{k}\right) \\
&= k\left(D^{T}\mathbf{u}_{i}\otimes\mathbf{e}^{k}\right) - k\left(A^{T}\mathbf{u}_{i}\otimes\mathbf{e}^{k}\right) \\
&= k\left(L^{T}\mathbf{u}_{i}\otimes\mathbf{e}^{k}\right) \\
&=k\lambda_{j}\left(\mathbf{u}_{i}\otimes\mathbf{e}^{k}\right),
\end{align*}
where $\lambda_{i}$ is the eigenvalue of $L^{T}$ corresponding to the eigenvector $\mathbf{u}_{i}$.
It follows from the mixed-product property of Kronecker producst that the eigenvectors $\left(\mathbf{u}_{i}\otimes\mathbf{e}^{k}\right)$ are orthogonal, for $i=1,\ldots,n-1$, and are perpendicular to $\mathbf{e}^{nk}=\left(\mathbf{e}^{n}\otimes\mathbf{e}^{k}\right)$.

Next, let $\mathbf{v}_{1},\ldots,\mathbf{v}_{k-1}$ denote the orthogonal eigenvectors of $J_{k}$ corresponding to the zero eigenvalue.
Also, let $\mathbf{e}_{1},\ldots,\mathbf{e}_{n}$ denote the column vectors of $I_{n}$, which are eigenvectors of $D^{T}$ corresponding to the eigenvalues $d^{+}(1),\ldots,d^{+}(n)$, respectively.
Then, for $i'=1,\ldots,n$ and $j=1,\ldots,k-1$, we have
\begin{align*}
\hat{L}^{T}\left(\mathbf{e}_{i'}\otimes\mathbf{v}_{j}\right) &= k\left(D^{T}\otimes I_{k}\right)\left(\mathbf{e}_{i'}\otimes\mathbf{v}_{j}\right) - \left(A^{T}\otimes J_{k}\right)\left(\mathbf{e}_{i'}\otimes\mathbf{v}_{j}\right) \\
&= k\left(D^{T}\mathbf{e}_{i'}\otimes\mathbf{v}_{j}\right) - \left(A^{T}\mathbf{e}_{i'}\otimes J_{k}\mathbf{v}_{j}\right) \\
&=k\left(d^{+}(i')\mathbf{e}_{i'}\otimes\mathbf{v}_{j}\right) - \left(A^{T}\mathbf{e}_{i'}\otimes 0\right) \\
&= kd^{+}(i')\left(\mathbf{e}_{i'}\otimes\mathbf{v}_{j}\right).
\end{align*}
Again, by the mixed-product property of Kronecker products, it follows that the eigenvectors $\left(\mathbf{e}_{i'}\otimes\mathbf{v}_{j}\right)$ are orthogonal, for $i'=1,\ldots,n$ and $j-1,\ldots,k-1$, and are perpendicular to $\mathbf{e}^{nk}$.
Moreover, since $\mathbf{v}_{1},\ldots,\mathbf{v}_{k-1}$ are perpendicular to $\mathbf{e}^{k}$, it follows that the eigenvectors
\[
\left(\mathbf{u}_{i}\otimes\mathbf{e}^{k}\right),~\left(\mathbf{e}_{i'}\otimes\mathbf{v}_{j}\right)
\]
are orthogonal, for $i=1,\ldots,n-1$, $i'=1,\ldots,n$, and $j=1,\ldots,k-1$.

Therefore, there are $nk-1$ orthogonal eigenvectors of $\hat{L}^{T}$ that are perpendicular to $\mathbf{e}^{nk}$, and Lemma~\ref{lem:rnrm-leigvecs} implies that $\hat{\Gamma}$ is restricted-normal.
Furthermore, by~\eqref{eq:rnrm-tprod1}, Proposition~\ref{prop:rnrml-djoin}, and Lemma~\ref{lem:imb-djoin}, $\Gamma$ is a directed join if and only if $\hat{\Gamma}$ is a directed join.
\end{proof}

We conclude this section by noting that the digraph in Figure~\ref{fig:rnrml-non-djoin} along with Theorems~\ref{thm:sqfree-rnrml}, ~\ref{thm:rnrm-ndjoin}, and~\ref{thm:rnrm-tprod} imply that there exist restricted-normal digraphs that are not directed joins if and only if the order is neither square-free nor twice a square-free number. 
Indeed, the natural numbers can be partitioned into square-free and non-square-free numbers.
Furthermore, the non-square-free numbers can be partitioned into the forms $kn^{2}$ and $4k$, where $n\geq 3$ and $k\geq 1$ is square-free.
Note that every number of the form $4k$ can be written as $8$ times an odd square-free number or twice an even square-free number.
All restricted-normal digraphs with an order that is square-free or twice a square-free number are directed joins by Theorem~\ref{thm:sqfree-rnrml}.
For all other possible orders a restricted-normal
digraph that is not a directed join can be constructed
either by using
Theorem~\ref{thm:rnrm-ndjoin} and Theorem~\ref{thm:rnrm-tprod}, in the case
of order $kn^{2}$ with $n\geq 3$ and $k\geq 1$ square-free, or by using
Figure~\ref{fig:rnrml-non-djoin} and Theorem~\ref{thm:rnrm-tprod},
in the case of order $8k$ where $k\geq 1$ is an odd square-free number.

\section{Conclusion}
The restricted numerical range is a novel tool for characterizing digraphs and studying their algebraic connectivity. 
In~\cite{Cameron2020_RNR}, digraphs with a restricted numerical range as a degenerate polygon, that is, a point or a line segment, are completely  described. 
In this article, we extended these results to include digraphs whose restricted numerical range is a non-degenerate convex polygon in the complex plane.
We refer to digraphs whose restricted numerical range is a degenerate or non-degenerate convex polygon in the complex plane as polygonal.

In Section~\ref{sec:comp}, we gave computational methods for finding polygonal digraphs, and showed that these digraphs can be partitioned into three classes: normal, restricted-normal, and pseudo-normal digraphs, all of which are closed under the digraph complement.
In Section~\ref{subsec:analysis-nrml}, we noted that normal digraphs include any digraph whose Laplacian matrix can be written as a circulant matrix.
Moreover, Proposition~\ref{prop:bal-nr-ops} and Theorem~\ref{thm:rnrml&bal} show that both the disjoint union and the bidirectional join of two normal digraphs results in another normal digraph.
Also, Theorem~\ref{thm:nrml-twin-splitting} provides a method for constructing normal digraphs that cannot be described as the disjoint union or bidirectional join of two (non-null) normal digraphs. 

Proposition~\ref{prop:bal-nr-ops} and Theorem~\ref{thm:rnrml&bal} also show that the directed join of two normal digraphs results in a restricted-normal digraph.
Moreover, Theorem~\ref{thm:sqfree-rnrml} shows that when the order is square-free or twice a square-free number, the only restricted-normal digraphs are directed joins of two normal digraphs.
Finally, Figure~\ref{fig:rnrml-non-djoin} along with Theorems~\ref{thm:rnrm-ndjoin} and~\ref{thm:rnrm-tprod} provide a construction for restricted-normal digraphs that are not directed joins when the order is neither square-free nor twice a square-free number.

Future research includes further investigation of the subclass of restricted-normal digraphs that are not directed joins and the class of pseudo-normal digraphs.



\end{document}